\documentclass[bj]{imsart}

\RequirePackage{amsthm,amsmath,amsfonts,amssymb}
\RequirePackage[numbers]{natbib}

\RequirePackage[colorlinks,citecolor=blue,urlcolor=blue]{hyperref}

\RequirePackage{mathtools}
\RequirePackage{tikz-cd}
\RequirePackage{float}
\RequirePackage{xcolor}
\RequirePackage{enumerate}
\RequirePackage{enumitem}
\RequirePackage{bbm}
\RequirePackage{subcaption}

\startlocaldefs
\numberwithin{equation}{section}
\theoremstyle{plain}
\newtheorem{theorem}{Theorem}[section]
\newtheorem*{theorem*}{Theorem}
\newtheorem{lemma}{Lemma}[section]
\newtheorem{proposition}{Proposition}[section]
\newtheorem{claim}{Claim}[lemma]

\theoremstyle{remark}

\newtheorem{remk}{Remark}[section]

\DeclareMathOperator{\1}{\mathbbm{1}}

\newcommand{\bel}{\begin{equation}\label}
\newcommand{\eel}{\end{equation}}
\newcommand{\be}{\begin{equation}}
\newcommand{\ee}{\end{equation}}
\newcommand{\bes}{\begin{equation*}}
\newcommand{\ees}{\end{equation*}}

\newcommand{\ninf}{n\to\infty}
\newcommand{\ninN}{n\in{\mathbb N}}
\newcommand{\ninNo}{n\in{\mathbb N}_0}

\DeclarePairedDelimiter\event{\Big\{}{\Big\}}
\DeclarePairedDelimiter\p{(}{)}

\newcommand{\ve}{\varepsilon}

\newcommand{\mbR}{{\mathbb R}}
\newcommand{\mbN}{{\mathbb N}}
\newcommand{\mn}{{\mathbb N}}

\newcommand{\R}{\mathbb{R}}
\newcommand{\mbE}{{\mathbb E}}

\newcommand{\mbP}{{\mathbb P}}

\newcommand{\cF}{{\cal F}}

\newcommand{\supp}{\supp{\rm supp}}

\newcommand{\E}{\mathbb{E}}
\newcommand{\Pb}{\mathbb{P}}

\newcommand{\tow}{\stackrel{{\rm d}}{\to}}
\newcommand{\toP}{\stackrel{\mbP}{\to}}

\endlocaldefs

\begin{document}

\begin{frontmatter}
\title{Functional limit theorems for random walks perturbed by positive alpha-stable jumps}
\runtitle{Functional limit theorems for random walks}

\begin{aug}
\author[A]{\fnms{Alexander} \snm{Iksanov}\ead[label=e1]{iksan@univ.kiev.ua}},
\author[B, C]{\fnms{Andrey} \snm{Pilipenko}\ead[label=e2,mark]{pilipenko.ay@gmail.com}}
\and
\author[D]{\fnms{ Ben} \snm{Povar}\ead[label=e3,mark]{oleksandr.prykhodko@warwick.ac.uk}}
\address[A]{Faculty of Computer Science and Cybernetics, Taras Shevchenko National University of Kyiv, Kyiv, Ukraine, \printead{e1}}

\address[B]{Department of the Theory of Stochastic Processes, Institute of Mathematics, National Academy of Sciences of Ukraine, Kyiv, Ukraine, \printead{e2}}
\address[C]{Department of Physics and Mathematics, National Technical University of Ukraine ``Igor Sikorsky Kyiv Polytechnic Institute'', Kyiv, Ukraine, \printead{e2}}

\address[D]{University of Warwick, Department of Statistics, Warwick, UK   \printead{e3}}
\end{aug}

\begin{abstract}
    Let $\xi_1$, $\xi_2,\ldots$ be i.i.d.\ random variables of zero mean and finite variance and $\eta_1$, $\eta_2,\ldots$ positive i.i.d.\ random variables whose distribution belongs to the domain of attraction of an $\alpha$-stable distribution, $\alpha\in (0,1)$. The two collections are assumed independent. We consider a Markov chain with jumps of two types. If the present position of the Markov chain is positive, then the jump $\xi_k$ occurs; if the present position of the Markov chain is nonpositive, then the jump $\eta_k$ occurs. We prove functional limit theorems for this and two closely related Markov chains under Donsker's scaling. The weak limit is a nonnegative process $(X(t))_{t\geq 0}$ satisfying a stochastic equation
    ${\rm d}X(t)={\rm d}W(t)+ {\rm d}U_\alpha(L_X^{(0)}(t))$, where $W$ is a Brownian motion, $U_\alpha$ is an $\alpha$-stable subordinator which is independent of $W$, and $L_X^{(0)}$ is a local time of $X$ at $0$. Also, we explain that $X$ is a Feller Brownian motion with a `jump-type' exit from $0$.
\end{abstract}


\end{frontmatter}


\section{Introduction and main result}

Let $x_0\in \mn_0:=\mn\cup\{0\}$ and $X(n)$ denote the number of claims at time $n\in\mn_0$ in a discrete time single server queuing model. The sequence
$X:=(X(n))_{n\in\mn_0}$ satisfies the classical Lindley recursion
\[
X(0)=x_0,\quad X(n)=(X(n-1)+\theta_{n})^+,\quad n\in\mn,
\]
where, as usual, $x^+=\max(x,0)$ for $x\in\mathbb{R}$, and a random variable $\theta_n$ represents `arrival minus departure at step $n$', see, for instance, Section III.6 in \cite{Asmussen:2003} or  Section 9.2 in \cite{Whitt:2002}. On the other hand, the sequence $X$ can be obtained as an action of the Skorokhod map $\Psi$ on
the random walk $S_\theta:=(S_\theta(n))_{n\in\mn_0}$ defined by $S_\theta(n):=x_0+\theta_1+\ldots+\theta_n$ for $n\in\mn_0$, namely,
\[
X= \Psi(S_\theta),
\]
see, for instance, Section 9.3.1 in \cite{Whitt:1980}. The scaling limit of $\Psi(S_\theta)$ is Skorokhod's reflection of the scaling limit of $S_\theta$, provided that the latter is well-defined. Assume, for instance, that $\theta_1$, $\theta_2,\ldots$ are independent identically distributed random variables of zero mean and finite variance. Then Donsker's scaling limit of $X$ is a reflected Brownian motion. The results of this type are well-known and can be interpreted from different viewpoints, see, for example, Chapter VI in \cite{Bertoin:1996}, Sections 1.9, 1.10 and 3.3 in \cite{Borovkov:1984}, Sections 8.7 and 8.8 in \cite{Whitt:2002}. Usually, the continuous mapping theorem is a main technical tool of the corresponding proofs.

Consider a slightly different model in which $X(n)$ is the number of goods at time $n$ in a storage. If $X(n-1)+\theta_n<0$ (that is, the request at time $n$ cannot be satisfied), then the storage is refilled with a random amount of goods or several random batches of goods. The purpose of the present paper is to prove functional limit theorems with Donsker's scaling for this and similar models, which only differ by the way of reflection upon crossing $0$. It will be shown below that if the distribution of the added number of goods belongs to the domain of attraction of an $\alpha$-stable distribution, $\alpha\in(0,1)$, then the heavy traffic limit is a reflected Brownian motion with infinite intensity jump-exit from $0$. In contrast to the classical heavy traffic limit theorems, the latter process is not a solution to the Skorokhod reflection problem.

Let $\xi$, $\xi_1$, $\xi_2,\ldots$ be i.i.d. real-valued random variables, $\eta$, $\eta_1$, $\eta_2,\ldots$ positive i.i.d. random variables and $(\tilde S_v(0))_{v>0}$, $(\hat S_v(0))_{v>0}$ and $(\grave S_v(0))_{v>0}$ families of random variables living on the same probability space; the three
collections being independent. For each $v>0$, define the random sequences $\tilde S_v:=(\tilde S_v(n))_{n\in\mn_0}$, $\hat S_v:=(\hat S_v(n))_{n\in\mn_0}$ and $\grave S_v:=(\grave S_v(n))_{n\in\mn_0}$ recursively as follows: for $n\in\mn_0$,
\be \label{tildeS}
\tilde S_v(n+1)=
\begin{cases}
\tilde S_v(n) + \xi_{n+1}, & \tilde S_v(n)>0, \\
\tilde S_v(n) + \eta_{n+1}, & \tilde S_v(n) \leq 0.
\end{cases}
\ee
\[
 \hat S_v(n+1)=
\begin{cases}
\hat S_v(n) + \xi_{n+1}, & \hat S_v(n) > 0, \\
\eta_{n+1}, & \hat S_v(n) \leq 0;
\end{cases}
\]
and
 \[
\grave S_v(n+1)=
\begin{cases}
\grave S_v(n) + \xi_{n+1}, & \grave S_v(n)>0 \mbox{ and } \grave S_v(n) + \xi_{n+1} > 0, \\
0, & \grave S_v(n)>0 \mbox{ and } \grave S_v(n) + \xi_{n+1} \leq 0,\\
\eta_{n+1}, & \grave S_v(n)=0
\end{cases}
\]

When $\tilde S_v(0)=\tilde S(0)$ for all $v>0$ and some random variable $\tilde S(0)$, we write $\tilde S$ for $\tilde S_v$.

The sequence $\tilde S$ is known in the literature as an oscillating random walk. The notion was introduced in \cite{Kemperman:1974} under no particular assumptions concerning the distributions of $\xi$ and $\eta$. Properties of general oscillating random walks and related models, in particular, recurrence/transience were investigated in \cite{Durrett+Kesten+Lawler:1991, Kemperman:1974, Menshikov+Petritis+Wade:2018, Rogozin+Foss:1978}. An explicit formula for $(z,t)\mapsto \sum_{n\geq 0}z^n\E e^{{\rm i}t\tilde S(n)}$, $|z|<1$, $t\in\R$ was given in \cite{Lotov:1996}. As far as we know, prior to our work there was just one functional limit theorem in this setting \cite{Helland:1981}. We discuss it in the paragraph following Theorem \ref{thm:main3}.

Denote by $D:=D[0,\infty)$ the Skorokhod space of c\`{a}dl\`{a}g functions defined on $[0,\infty)$.
We assume that the space $D$ is endowed with the $J_1$-topology and write $\Rightarrow$ for weak convergence in this space. Also, on several occasions, we denote by $\Rightarrow$ weak convergence in $D^k$ for $k\geq 2$ equipped with the product $J_1$-topology. In the latter case the topology and the space will be specified. Comprehensive information concerning the $J_1$-topology can be found in the books
\cite{Billingsley:1999, Jacod+Shiryaev:2003}. As usual, $\overset{\Pb}\to$ denotes convergence in probability, $\lfloor x\rfloor$ denotes the integer part of $x\in\R$ and $\circ$ denotes composition of (random) functions. Recall that the Euler gamma function $\Gamma$ is given by $\Gamma(x):=\int_0^\infty e^{-y}y^{x-1}{\rm d}y$ for $x>0$.

Here is our main result.
\begin{theorem}\label{thm:main}
Assume that $\E \xi = 0$, $\sigma^2:={\rm Var}\,\xi\in (0,\infty)$ and that
\begin{equation}\label{eq:reg}
\Pb\{\eta>z\}\sim z^{-\alpha}\ell(z),\quad z\to\infty
\end{equation}
for some $\alpha\in (0,1)$ and {some} $\ell$ slowly varying at $\infty$.
  If the initial values satisfy
$$\frac{\tilde S_v(0)}{v^{1/2}}~\overset{\Pb}\to~ x,\quad v\to\infty$$
for some $x\geq 0$, then
\begin{equation}\label{eq:impo}
\Big(\frac{\tilde S_v(\lfloor vt\rfloor)}{\sigma v^{1/2}}\Big)_{t\geq 0}~\Rightarrow~ (W_\alpha(x,t))_{t\geq 0},
\quad v\to\infty,
\end{equation}
where the limit process $W_\alpha^{(x)}:=(W_\alpha(x,t))_{t\geq 0}$ is given by
\bel{eq:representation}
W_\alpha(x,t)=x+W(t)+U_\alpha \circ U_\alpha^\leftarrow \circ ((-x+M(t))^+),\quad t\geq 0
\ee
Here $W:=(W(t))_{t\geq 0}$ is a standard Brownian motion;
$$M(t) = -\min_{s\in [0,\,t]}\,W(s),\quad t\geq 0,$$ $U_\alpha:=(U_\alpha(t))_{t\geq 0}$ is a drift-free $\alpha$-stable subordinator independent of $W$
with
\begin{equation}\label{eq:1}
\E\exp(-zU_\alpha(t))=\exp(-\Gamma(1-\alpha)tz^\alpha),\quad t,z  \geq 0,
\end{equation}
and $U_\alpha^\leftarrow:=(U_\alpha^\leftarrow(t))_{t\geq 0}$ is an inverse $\alpha$-stable
subordinator defined by $$U_\alpha^\leftarrow (t)=\inf\{s\geq 0: U_\alpha(s) >t\},\quad t\geq 0.$$
\end{theorem}

The limit process $W^{(x)}_\alpha$ is rather non-standard. This statement is justified in Section \ref{sec:properties}. As an appetizer, we only mention here that $W^{(x)}_\alpha$ is a {\it Feller Brownian motion} with a `jump-type' exit from $0$.

Theorem \ref{thm:main3} states that the result of Theorem \ref{thm:main} continues to hold, with the sequence $\tilde S$ replaced by either $\hat S$ or $\grave S$.
\begin{theorem}\label{thm:main3}
 Assume that $\E \xi = 0$, $\sigma^2={\rm Var}\,\xi\in (0,\infty)$ and that condition \eqref{eq:reg} holds.
If the initial values satisfy $$\frac{\hat S_v(0)}{v^{1/2}}~\overset{\Pb}\to~ x,\quad v\to\infty$$ for some $x\geq 0$, then
\[
\Big(\frac{\hat S_v(\lfloor vt\rfloor)}{\sigma v^{1/2}}\Big)_{v\geq 0}~\Rightarrow~ (W_\alpha(x,t))_{t\geq 0},
\quad v\to\infty.
\]
If $$\frac{\grave S_v(0)}{v^{1/2}}~\overset{\Pb}\to~ x,\quad v\to\infty$$ for some $x\geq 0$, then
\[
  \Big(\frac{\grave S_v(\lfloor vt\rfloor)}{\sigma v^{1/2}}\Big)_{t\geq 0}~\Rightarrow~
(W_\alpha(x,t))_{t\geq 0},\quad v\to\infty.
  \]
\end{theorem}

We note in passing that the case when $\eta$ has a finite mean is much easier to deal with. The scaling limit for $\tilde S$, $\hat S$ and $\grave S$ is then a reflected Brownian motion, see \cite{NgoPeigne, PilipenkoPrykhodko2020}.
Corollary 8.4 in \cite{Helland:1981} is a functional limit theorem for $(v^{-1/2}\bar S (\lfloor vt\rfloor))_{t\geq 0}$ as $v\to\infty$ with the weak limit being an oscillating Brownian motion. Here, $(\bar S(n))_{n\in\mn_0}$
 differs from $(\tilde S(n))_{n\in\mn_0}$ defined in \eqref{tildeS} in that $\bar S(0)=0$ and, for $n\in\mn_0$, $\bar S(n+1):=\bar S(n)+\omega_{n+1}$ provided that $\bar S(n)=0$, and $\omega$, $\omega_1,\ldots$ are i.i.d. real-valued random variables which are independent of $(\xi_k)_{k\in\mn}$ and $(\eta_j)_{j\in\mn}$,  $\E \xi=\E \eta=\E\omega=0$, ${\rm Var}\,\xi<\infty$, ${\rm Var}\,\eta<\infty$ and ${\rm Var}\,\omega\in (0,\infty)$.

We believe that, at the expense of much heavier machinery, the assumption of positivity of $\eta$ could have been relaxed. It seems that Theorems \ref{thm:main} and \ref{thm:main3} should continue to hold whenever the distribution of $\eta$ (possibly taking values of both signs) belongs to the domain of attraction of an $\alpha$-stable distribution, $\alpha\in (0,1)$.

Some ideas of the present work borrow heavily from It\^{o}'s excursion theory as presented in the book \cite{Blumenthal:1992}.
In particular, our argument is very different from that exploited in \cite{Helland:1981}. To make the link visible, observe that the excursions between consecutive crossings of zero of the Markov chain under consideration coincide with the excursions of a random walk driven by $\xi$. Since $\xi$ has a finite second moment, these excursions should be close in some sense to those of a Brownian motion. Thus, the limit process has to behave like a Brownian motion in the upper half-plane. The distribution of the jumps $\eta_1$, $\eta_2,\ldots$ into the positive halfline belongs to the domain of attraction of a stable distribution on $[0,\infty)$. In particular, the sum of these jumps, properly scaled, converges weakly to a stable subordinator. The additional contribution to the limit process is, roughly speaking, made by the composition of the sum of jumps and the number of crossings of $0$ up to time $n$. Even though neither of the composed processes converges weakly under Donsker's scaling, their composition does indeed exhibit growth at the `magic' square-root rate.

In the article \cite{Pilipenko+Prykhodko:2014} a functional limit theorem similar to ours is proved
in a much simpler situation where $\xi$ and $\eta$ are integer-valued and $\xi$ is bounded from below by $-1$. As far as we know there are no other functional limit theorems that would make an explicit link between random walks with a random-jump reflection at $0$ and  $W^{(x)}_\alpha$. Also relevant to the present work are the papers
\cite{LambertSimatos2014, Yano2008, Yano2015} and references therein, in which functional limit theorems are obtained for processes merged together from certain excursions.

The authors of \cite{Pilipenko+Prykhodko:2014} invoke a representation arising in a generalized Skorokhod reflection problem (see \cite{Pilipenko:2012} for more details concerning this problem) as a principal tool. The approach of the cited paper fails in the present setting of real-valued random variables $\xi$ and $\eta$, for no reduction to the generalized Skorokhod reflection problem seems to be possible. As a remedy, we offer a novel argument which forms the main achievement of the paper.

The remainder of the paper is structured as follows. After some preliminary work Theorems \ref{thm:main} and \ref{thm:main3} are proved in Section \ref{sec:pro1} with the help of more general Theorem \ref{thm:main_eta}. The proof of Theorem \ref{thm:main_eta} is given in Section \ref{sec:proof}. Properties of the limit process are investigated in Section \ref{sec:properties}. Finally, the \hyperref[Appendix]{Appendix} collects auxiliary results concerning the $J_1$-convergence of deterministic functions.

\section{Proofs of Theorems \ref{thm:main} and \ref{thm:main3}}\label{sec:pro1}

\subsection{Preliminary discussion}

Let $X_1$, $X_2,\ldots$ be independent copies of a real-valued random variable $X$. Throughout the paper, we adopt generic notation $S_X:=(S_X(n))_{n\in\mn}$ for a random walk with increments
$X_k$, that is, $$S_X(n):=X_1+\ldots+X_n,\quad n\in\mn.$$ Also, we put $$\nu_X(t):=\inf\{k\in\mn: S_X(k)>t\},\quad t\geq 0,$$ so that $(\nu_X(t))_{t\geq 0}$ is the first-passage time process for $(S_X(n))_{n\in\mn}$.

By Donsker's theorem, $$\Big(\frac{S_\xi(\lfloor vt\rfloor)}{\sigma v^{1/2}}\Big)_{t\geq 0} ~\Rightarrow~ (W(t))_{t\geq 0},\quad v\to\infty.$$
Condition \eqref{eq:reg} ensures
\begin{equation}\label{eq:sub}
\Big(\frac{S_\eta(\lfloor vt\rfloor)}{a(v)}\Big)_{t\geq 0}~\Rightarrow~ (U_\alpha(t))_{t\geq 0},
\quad v\to\infty,
\end{equation}
where $a:[0,\infty)\to (0,\infty)$ is any function satisfying $\lim_{v\to\infty}v\Pb\{\eta>a(v)\}=1$.
{ Equivalently,
\begin{equation}\label{eq:sub2}
\Big(\frac{S_\eta(\lfloor c(\sigma^2 v)t\rfloor)}{\sigma v^{1/2}}\Big)_{t\geq 0}~\Rightarrow~ (U_\alpha(t))_{t\geq 0},
\quad v\to\infty,
\end{equation}
where $c$ is generalized inverse of $a^2$.} Furthermore, in view of independence of $\xi$ and $\eta$,
\begin{equation}\label{eq:joint2}
\Big(\frac{S_\xi(\lfloor vt\rfloor)}{\sigma v^{1/2}}, \frac{S_\eta(\lfloor c(\sigma^2 v)t\rfloor)}{\sigma v^{1/2}}\Big)_{t\geq 0}~\Rightarrow~ (W(t),U_\alpha(t))_{t\geq 0},\quad v\to\infty
\end{equation}
in $D^2$ in the product $J_1$-topology, where $W$ and $U_\alpha$ are assumed independent.

\subsection{Reduction to the case of nonpositive starting point and $x=0$}

Below we explain that, without loss of generality, we can assume that the starting point $x$ of the limit processes is $0$ and that $\tilde S_v(0)$, $\hat S_v(0)$ and $\grave S_v(0)$ are a.s.\ nonpositive. We only provide a detailed argument for $\tilde S_v$, for the reasoning for the other two processes is completely analogous.

If $\tilde S_v(0)\leq 0$ a.s., then necessarily $x=0$. Assume that $\tilde S_v(0)>0$ a.s. As a preparation for what follows, we formulate Proposition \ref{corl:excursion_parts} which follows from Proposition \ref{lemma:excursion_parts} given in the Appendix and Skorokhod's representation theorem.
\begin{proposition}\label{corl:excursion_parts}
For $n\in\mn_0$, let $ X_n$ be a strong Markov process and $\tau_n$ an a.s.\ finite stopping time with respect to the natural filtration of $X_n$. Assume that
\begin{itemize}
\item $(X_n(\cdot\wedge \tau_n), \tau_n) \Rightarrow (X_0(\cdot\wedge \tau_0), \tau_0)$, $n\to\infty$, in the product $J_1$-topology in $D^2$;

\item $X_n(\cdot+\tau_n)\Rightarrow X_0(\cdot+ \tau_0)$, $n\to\infty$ in $D$.
\end{itemize}

\noindent Then $X_n\Rightarrow  X_0$, $n\to\infty$ in $D$.
\end{proposition}
Let $(v_n)_{n\in\mn}$ be any sequence of positive numbers which diverges to $+\infty$ as $n\to\infty$. We intend to apply Proposition \ref{corl:excursion_parts} with $X_n(\cdot):= \tilde S_{v_n}(\lfloor v_n\cdot\rfloor)/(\sigma v_n^{1/2})$, $X_0:=W_\alpha^{(x)}$ and $\tau_n:=v_n^{-1}T_{v_n}$, where $T_{v_n}$ is the first entrance time of $\tilde S_{v_n}$ into $(-\infty, 0]$. For each $x\geq 0$, the process $W_\alpha^{(x)}$ is a strong Markov process, see  Theorem 3.11 in Chapter II of \cite{Blumenthal:1992}, and, for $x\neq y$, $W_\alpha^{(x)}$ and $W_\alpha^{(y)}$ have the same transition probabilities. It can be seen from the definition that $W_\alpha^{(x)}(t\wedge \tau_0)=x+W(t\wedge \tau_0)$, where $W$ is a Brownian motion, that is, $W_\alpha^{(x)}$ behaves as a Brownian motion until hitting $0$. It is known that, with this choice, the first limit relation of Proposition \ref{corl:excursion_parts} holds. The process with translated time $(\tilde S_{v_n}(\lfloor v_n t\rfloor +T_{v_n}))_{t\geq 0}$ has the same distribution as $(\tilde S_{v_n}(\lfloor v_n t\rfloor))_{t\geq 0}$ in which $\tilde S_{v_n}(0)$ has the same distribution as $\tilde S_{v_n}(T_{v_n})\leq 0$ a.s. and is assumed independent of $(\xi_k)_{k\in\mn}$ and $(\eta_j)_{j\in\mn}$. According to Proposition \ref{corl:excursion_parts}, we can thus identify the two processes in the subsequent proof.

Left with showing that
\begin{equation}\label{eq:2}
\frac{\tilde S_{v_n}(T_{v_n})}{v_n^{1/2}}\toP 0,\quad n\to\infty
\end{equation}
we fix any $b>0$ and note that on the event $\{T_{v_n}\leq bv_n\}$
$$\Big|\frac{\tilde S_{v_n}(T_{v_n})}{v_n^{1/2}}\Big| \leq \frac{\max_{k\leq bv_n+1}\,|\xi_k|}{v_n^{1/2}}\quad\text{a.s.}$$
(here, the present assumption $\tilde S_{v_n}(0)> 0$ a.s.\ plays a crucial role). Further, the relation
\bel{maxXi1}
\frac{\max_{k\leq bv_n+1}\,|\xi_k|}{v_n^{1/2}}\toP 0,\quad n\to\infty
\ee
is a consequence of $\E \xi^2<\infty$. With these at hand, write, for any $\varepsilon>0$ and any $b>0$,
\begin{multline*}
\Pb\{|\tilde S_{v_n}(T_{v_n})|>\varepsilon v_n^{1/2}\}=\Pb\{\ldots, T_{v_n}\leq bv_n\}+\Pb\{\ldots, T_{v_n}>bv_n\}\\ \leq \Pb\{\max_{k\leq bv_n+1}\,|\xi_k|>\varepsilon v_n^{1/2}\Big\}+\Pb\{T_{v_n}>bv_n\}.
\end{multline*}
Hence, $${\lim\sup}_{n\to\infty} \Pb\{|\tilde S_{v_n}(T_{v_n})|>\varepsilon v_n^{1/2}\}\leq \Pb\{\tau_0>b\},$$ where $\tau_0$ is the first entrance time of $W_\alpha^{(x)}$ into $(-\infty, 0]$. Since $\tau_0$ is a.s.\ finite, we arrive at \eqref{eq:2} on letting $b\to\infty$.

\subsection{Passage to an equivalent model}

Recall that we assume that $\tilde S_v(0)\leq 0$ a.s.\ and that $x=0$. For notational simplicity, we shall omit the index $v$ in the notation, if there is no ambiguity. Also, we shall write $W_\alpha(t)$ for $W_\alpha(0,t)$.

Consider a possible realization of the first five elements of the sequence $\tilde S$: $\tilde S(0)\leq 0$,
$\tilde S(1)= \tilde S(0)+\eta_1\leq 0$, $\tilde S(2)=\tilde S(0)+\eta_1+\eta_2> 0$, $\tilde S(3)=\tilde S(0)+\eta_1+\eta_2+\xi_3>0$,
 $\tilde S(4)=\tilde S(0)+\eta_1+\eta_2+\xi_3+\xi_4\leq 0$, $\tilde S(5)=\tilde S(0)+\eta_1+\eta_2+\xi_3+\xi_4+\eta_5$.
We observe that the variables $\xi_1$, $\xi_2$, $\eta_3,$  $\eta_4$, and $\xi_5$
 are missing in this realization. More generally, for a given $k\in\mn$ any particular
 realization involves either $\xi_k$ or $\eta_k$ but not both. Thus, the presence of missing variables
 is an intrinsic feature of the model. We prefer to work with an equivalent model whose construction
 uses the whole collection $(\xi_k,\eta_k)_{k\in\mn}$ with no gaps. To this end, we define a new
sequence $\tilde {\tilde S}:=(\tilde {\tilde S}(n))_{n\in\mn_0}=(\tilde {\tilde S}_v (n))_{n\in\mn_0}$
 and an auxiliary sequence $(\tilde  T(n))_{n\in\mn_0}$ by
$$\tilde {\tilde S}(0):=\tilde S_v(0),\quad \tilde T(0):=1$$ and, for $n\in\mn$,
\bel{eq:sequence_equiv}
\tilde {\tilde S}(n+1):=
\begin{cases}
\tilde {\tilde S}(n) + \xi_{n+1-\tilde T(n)}, & \tilde {\tilde S}(n) > 0, \\
\tilde {\tilde S}(n) + \eta_{\tilde T(n)}, & \tilde {\tilde S}(n) \leq 0
\end{cases}
\ee
and
$$\tilde T(n+1):=1+\#\{1\leq k\leq n: \tilde {\tilde S}(k) \leq 0\}.$$ Observe that
\begin{equation}\label{eq:repr10}
\tilde {\tilde S}(n)= S_\xi(n - \tilde T(n)) + S_\eta(\tilde T(n)),\quad n\in\mn_0.
\end{equation}
The sequences $\tilde {\tilde S}$ and $\tilde S$ have the same distribution. In view of this Theorem \ref{thm:main} follows if we can prove a counterpart of \eqref{eq:impo}:
\begin{equation}\label{eq:count3}
\Big(\frac{\tilde {\tilde S}(\lfloor vt\rfloor)}{\sigma v^{1/2}}\Big)_{t\geq 0}~\Rightarrow~
(W_\alpha(t))_{t\geq 0},
\quad v\to\infty.
\end{equation}

Further we formulate and prove a general result and obtain Theorems \ref{thm:main} and \ref{thm:main3} as corollaries. Let $\zeta_1$, $\zeta_2,\ldots$ be (possibly dependent) positive random variables {which may depend on $(\xi_k)_{k\in\mn}$, $\tilde S_v(0)$ and/or $\hat S_v(0)$.}

Put $$\mathcal{S}_\zeta(n):=\zeta_1+\ldots+\zeta_n,\quad n\in\mn.$$ For $v>0$, let $S_v^\ast=(S_v^\ast(n))_{n\in\mn_0}$ be a sequence satisfying $S_v^\ast(0)=\tilde S_v(0)$ or $S_v^\ast(0)=\hat S_v(0)$ and
\begin{equation}\label{eq:repr}
S_v^\ast(n)= S_\xi(n - T_v(n)) +\mathcal{S}_\zeta(T_v(n)),\quad n\in\mn_0,
\end{equation}
where $T_v(0):=1$ and, for $n\in\mn$, $$T_v(n+1):=1+\#\{1\leq k\leq n: S_v^\ast(k) \leq 0\}.$$
\begin{theorem}\label{thm:main_eta}
Assume that $\E \xi = 0$, $\sigma^2:={\rm Var}\,\xi\in (0,\infty)$,
\bel{eq:492}
\lim_{n\to\infty}\frac{ {S_{|\xi|}(n)}}{\mathcal{S}_\zeta(n)}=0 \quad {\rm a.s.}
\ee
and that, for some positive function $b$,
\begin{equation}\label{eq:conv_W_AND_U}
\Big(\frac{S_\xi(\lfloor vt\rfloor )}{\sigma v^{1/2} }, \frac{\mathcal{S}_\zeta(\lfloor b(v)t\rfloor)}{{\sigma}v^{1/2}}\Big)_{t\geq 0}~\Rightarrow~ (W(t), U_\alpha(t))_{t\geq 0},\quad v\to\infty
\end{equation}
in the product $J_1$-topology in $D$, where $W$ and $U_\alpha$ are as in Theorem \ref{thm:main}. Assume that $S^\ast_v(0)\leq 0$ a.s. and
\bel{eq:490}
\frac{ S^\ast_v(0)}{v^{1/2}}\toP 0,\quad v\to\infty.
\ee
Then
\begin{equation}\label{eq:impo1}
\Big(\frac{ S^\ast_v(\lfloor vt\rfloor)}{\sigma v^{1/2}}\Big)_{t\geq 0}~\Rightarrow~ (W_\alpha(t))_{t\geq 0},
\quad v\to\infty.
\end{equation}
\end{theorem}
The proof of Theorem \ref{thm:main_eta} will be given in Section \ref{sec:proof}.

\subsection{Proof of Theorem \ref{thm:main}}

In the setting of Theorem \ref{thm:main} we apply Theorem \ref{thm:main_eta} with $\zeta_k=\eta_k$, $S^\ast=\tilde {\tilde S}$ (see \eqref{eq:repr10} and \eqref{eq:repr}). Since relation \eqref{eq:492} holds according to the strong law of large numbers for random walks, and \eqref{eq:conv_W_AND_U}, {with $b(x)=c(\sigma^2 x)$}, is nothing else but \eqref{eq:joint2}, Theorem \ref{thm:main} follows from Theorem \ref{thm:main_eta}.

\subsection{Proof of Theorem \ref{thm:main3}}

By the definitions of $\hat S$ and $\grave S$,
\[
\max_{0\leq k\leq n}|\hat S(k)-\grave S(k)| \leq \max_{0\leq k\leq n}|\xi_k|,\quad n\in\mn\quad \text{a.s.}
\]
In view of \eqref{maxXi1} and Slutsky's lemma we conclude that if one of the processes $(\frac{\hat S_v(\lfloor vt\rfloor)}{\sigma v^{1/2}})_{t\geq 0}$ and $(\frac{\grave S_v(\lfloor vt\rfloor)}{\sigma v^{1/2}})_{t\geq 0}$ converges weakly, then so does the other and the weak limits of the processes are the same. Hence, we only treat $(\frac{\hat S_v(\lfloor vt\rfloor)}{\sigma v^{1/2}})_{t\geq 0}$.

While doing so we apply Theorem \ref{thm:main_eta} with a particular choice of $(\zeta_n)$ that we are now going to explain. To this end, define sequences ${\hat {\hat S}}=({\hat {\hat S}}(n))_{n\in\mn_0}$ and $(\hat T(n))_{n\in\mn_0}$ by
$${\hat {\hat S}}(0):=\hat S_v(0),\quad   \hat T(0):=1$$ and, for $n\in\mn$,
$${\hat {\hat S}}(n+1):=
\begin{cases}
{\hat {\hat S}}(n) + \xi_{n+1-\hat T(n)}, & {\hat {\hat S}}(n) > 0, \\
\eta_{\hat T(n)}, & {\hat {\hat S}}(n) \leq 0
\end{cases}$$
and
$$\hat T(n+1):=1+\#\{1\leq k\leq n: {\hat {\hat S}}(k) \leq 0\}.$$ Put $\Theta_1 = 0$,
\begin{equation}\label{eq:deftheta}
\Theta_i:= \inf\{l > \Theta_{i-1}: {\hat {\hat S}}(l) \leq 0 \},\quad i \geq 2
\end{equation}
and then, for $i\in\mbN$,
\be \label{zetaAndGamma}
\begin{split}
\gamma_i &:= -{\hat {\hat S}}(\Theta_{i}), \\
\zeta_{i} &:= {\hat {\hat S}}(\Theta_{i} + 1) + \gamma_i =\eta_i + \gamma_i.
\end{split}
\ee
The random variable $\Theta_i$ is the time of the $i$-th visit of ${\hat {\hat S}}$ to $(-\infty,0]$. The random variable $\gamma_i$ is the {\it $(i-1)$-th overshoot of ${\hat {\hat S}}$ into $(-\infty,0]$}. Observe that since the random variables $\gamma_1$, $\gamma_2,\ldots$ are i.i.d., so are $\zeta_1$, $\zeta_2,\ldots$. A.s.\ nonnegativity of $\gamma_i$ entails $\zeta_i \geq \eta_i$ a.s. Furthermore, $\zeta_i=\eta_i$ a.s.\ if and only if $\gamma_i = 0$ a.s. This simpler situation of zero overshoot into $(-\infty,0]$ occurs in the setting of \cite{Pilipenko+Prykhodko:2014}, where $\xi$ is an integer-valued random variable with $\xi \geq -1$ a.s., and $\eta$ is a positive integer-valued random variable.

Since $\hat {\hat S}$ satisfies a counterpart of \eqref{eq:repr}
$$\hat {\hat S}(n)=S_\xi(n-\hat T(n))+S_\zeta(\hat T(n)),\quad n\in\mn_0$$
we conclude that $\hat {\hat S}$ is a particular instance of $S^\ast=S^\ast_v$ treated in Theorem \ref{thm:main_eta}.
According to Theorem \ref{thm:main_eta}, we are left with showing that \eqref{eq:492} and \eqref{eq:conv_W_AND_U} hold forour particular choice of $(\zeta_i)$. Since $\E|\xi|<+\infty$ and $\E\eta=+\infty$, we infer
$$0\leq \frac{S_{|\xi|}(n)}{S_\zeta(n)}\leq \frac{S_{|\xi|}(n)}{S_\eta(n)}~\to~ 0,\quad n\to\infty \quad\text{a.s.}$$
by the strong law of large numbers for random walks.
This proves \eqref{eq:492}. Note that the sequence $(\zeta_i)$
depends on $(\xi_k)$. This is allowed in the definition of $(\zeta_i)$.

To check \eqref{eq:conv_W_AND_U}, we need an auxiliary result, Lemma \ref{lem:ratio}, which shows that the standard random walks $S_\eta$ and $S_\zeta$ behave similarly which particularly means that the contribution of the sum of the overshoots is negligible in comparison to $S_\eta$.
\begin{lemma}\label{lem:ratio}
Under the assumptions of Theorem \ref{thm:main3},
\be
\frac{S_\zeta(n)}{S_\eta(n)} \toP 1, \ \ninf.
\ee
\end{lemma}
\begin{proof}
Put $\tau_0:= 0$ and, for $i\in\mn$, $$\tau_{i+1}:=\inf\{k> \tau_i\ :\ S_{\xi}(k)<S_{\xi}(\tau_i)\}\quad\text{and}\quad \chi_i:=S_{\xi}(\tau_{i-1}) - S_{\xi}(\tau_{i}).$$
The elements of the sequences $(\tau_i)_{i\in\mn}$ and $(\chi_i)_{i\in\mn}$ are called {\it descending ladder epochs} and {\it descending ladder heights} of $S_\xi$, respectively. By construction $\chi>0$ a.s. Further, the assumptions $\mbE\xi=0$ and ${\rm Var}\,\xi < \infty$ entail $\mu:=\mbE \chi<\infty,$ see, for instance, formula (4b) in \cite{Doney:1980}. Recall $(\gamma_i)_{i\in\mbN}$ from \eqref{zetaAndGamma} and note that (except for $\gamma_1$ which is $0$ a.s.) these are independent copies of a random variable $\gamma$ with
$$\gamma \overset{{\rm d}}{=} S_\chi(\nu_\chi(\eta))-\eta.$$

We start by showing that
\be\label{iksan2}
\lim_{x\to\infty}\frac{\mbP\{\gamma > x\}}{\mbP\{\eta > x\}} = 0.
\ee
Denote by $F_\chi$ the distribution function of $\chi$ and $U_\chi$ the renewal function for $(S_\chi(n))_{n\in\mn_0}$, that is, $U_\chi(x):= \sum_{n \geq 0} \mbP\{S_\chi(n) \leq x\}$ for $x\in\R$.
Then
$$\mbP\{S_\chi(\nu_\chi(z))-z > x\}=\int_{[0,\,z]} \left(1-F_\chi(z + x - y) \right){\rm d}U_\chi(y),\quad  z,x\geq 0.$$
Further, for any $A > 0$,
\be\label{eq:380}
\begin{split}
\mbP\{\gamma > x\} &=\mbE \int_{[0,\,\eta]}\left( 1- F_\chi(\eta + x - y)\right){\rm d}U_\chi(y) \\
&= \mbE \int_{[0,\,\eta]}(1- F_\chi(\eta + x - y)){\rm d}U_\chi(y)(\1_{\{\eta \leq Ax\}} + \1_{\{\eta  > Ax\}})  \\
&\leq (1-F_\chi(x))\mbE U_\chi(\eta)\1_{\{\eta\leq Ax\}}+ \mbP\{\eta > Ax\}
\end{split}
\ee
having utilized monotonicity of $F_\chi$ for the inequality. By the elementary renewal theorem, $\lim_{x\to\infty}x^{-1} U_\chi(x)=\mu^{-1}$ and thereupon
\begin{equation*}
\frac{\mbE U_\chi(\eta)\1_{\{\eta\leq Ax\}}}{\mbP\{\eta>x\}}~\sim~
\frac{\mbE \eta \1_{\{\eta\leq Ax\}}}{\mu \mbP\{\eta>x\}},\quad x\to\infty.
\end{equation*}
Recalling \eqref{eq:reg} and invoking Karamata's theorem (Theorem 1.6.4 in \cite{Bingham+Goldie+Teugels:1989}) we infer
$$\frac{\mbE  \eta \1_{\{\eta\leq Ax\}}}{\mu \mbP\{\eta>x\}}~\sim~\frac{\alpha A^{1 - \alpha}}{(1 - \alpha)\mu}x,\quad x\to\infty.$$
This yields
\be\label{eq:401}
\lim_{x\to\infty}\frac{(1-F_\chi(x))\mbE U_\chi(\eta)\1_{\{\eta\leq Ax\}}}{\mbP\{\eta>x\}}=0,
\ee
because $\mbE \chi<\infty$ entails $\lim_{x\to\infty}(1-F_\chi(x))x=0$.

It follows from \eqref{eq:380} and \eqref{eq:401} that, for any $A>0$,
\[\limsup_{x\to\infty}\frac{\mbP\{\gamma > x\}}{\mbP\{\eta > x\}}\leq \lim_{x\to\infty}\frac{\mbP\{\eta > Ax\}}{\mbP\{\eta > x\}}=A^{-\alpha}.
\]
Since $A>0$ is arbitrary, we arrive at \eqref{iksan2}. Thus, we have proved that given $\varepsilon>0$ there exists $x_0>0$ such that
\[\mbP\{\gamma >x\} \leq \ve \mbP\{\eta >x\}\] whenever $x\geq x_0$. Let $\hat \eta$ be a random variable with distribution \[\mbP\{\hat \eta > x\} =
\begin{cases}
1, \ &x < x_0, \\
\ve \mbP\{\eta>x\}, \ &x \geq x_0.
\end{cases}\]
Then $\mbP\{\gamma> x\} \leq \mbP\{\hat \eta > x\}$ for $x\geq 0$ and, as a consequence, for each $n\in\mn$,
\be\label{eq:432}
\mathbb{P}\{\gamma_1+\ldots+\gamma_n>x\}\leq \mathbb{P}\{\hat \eta_1+\ldots+\hat \eta_n>x\},\quad x\geq 0,
\ee
where $\hat \eta_1, \hat \eta_2,\ldots$ are independent copies of $\hat \eta$. Since $\mathbb{P}\{\hat \eta > x\}\sim \varepsilon x^{-\alpha}\ell(x)$ as $x\to\infty$, we conclude that a counterpart of \eqref{eq:sub} holds for $(S_{\tilde \eta}(n))_{n\in\mn_0}$. Its specialization to $v=1$ reads
\be\label{eq:434}
\frac{S_{\tilde \eta}(n)}{a_n}~\tow~ \ve^{1/\alpha} U_\alpha(1),\quad n\to\infty.
\ee
Since $\ve>0$ is arbitrary, we deduce from \eqref{eq:432}, \eqref{eq:434} and one-dimensional version of \eqref{eq:sub}
that
\[
\frac{S_\gamma(n)}{S_\eta(n)} \toP 0,\quad n\to\infty.
\]
This completes the proof of Lemma \ref{lem:ratio}.
\end{proof}

{We shall show that \eqref{eq:conv_W_AND_U} holds with $b(x)=c(\sigma^2 x)$, where $c$ is the same as in \eqref{eq:sub2}.} Recall the definition of $\zeta$ from \eqref{zetaAndGamma} and note that the process $(S_\zeta(\lfloor t\rfloor)-S_\eta(\lfloor t\rfloor))_{t\geq 0}$ is a.s.\ nondecreasing. Hence,
for all $T>0$,
\be\label{eq:469}
\sup_{t \in[0,\,T]} \Big|\frac{S_\zeta(\lfloor b(v)t\rfloor)}{v^{1/2}}-
\frac{S_\eta(\lfloor b(v)t\rfloor)}{v^{1/2}}\Big|\leq
 \Big|\frac{S_\zeta(\lfloor b(v)T\rfloor)}{S_\eta(\lfloor
b(v)T\rfloor)}-1\Big|\cdot  \frac{S_\eta(\lfloor b(v)T\rfloor)}{v^{1/2}}~\toP~ 0,\quad v\to\infty,
\ee
where the limit relation is a consequence of Lemma \ref{lem:ratio} and one-dimensional version
of {\eqref{eq:sub2}}. This together with  \eqref{eq:joint2} proves \eqref{eq:conv_W_AND_U}. The proof of Theorem \ref{thm:main3} is complete.

\section{Proof of Theorem \ref{thm:main_eta}}\label{sec:proof}

For $n\in\mn_0$, put
\[m(n) := -\min_{0\leq k \leq n}(S^\ast_v(0)+ S_\xi(k)) \]
and
\be\label{Rdefinition}
R(n):=S^\ast_v(0)+S_\xi(n)+\mathcal{S}_\zeta \circ \bar \nu_\zeta\circ m(n),
\ee
where $\bar \nu_\zeta(t):=\inf\{k\in\mn: \mathcal{S}_\zeta(k)>t\}$ for $t\geq 0$.

Now we explain how the rest of the proof is organized. We start by proving in Lemma \ref{lemma:timeChange} that the sequence $R:=(R(n))_{n\in\mn_0}$ can be obtained
from $S^\ast$ by a time-change.
Lemma \ref{lemma:RConvergence} states that relation \eqref{eq:count3} holds with $R$ replacing $S^\ast$. Finally, Lemma \ref{lemma:lambda_id} makes it clear that the time-change defined in Lemma \ref{lemma:timeChange} is close to the identity mapping. Combining all these auxiliary results we arrive at \eqref{eq:count3}, thereby completing the proof of Theorem \ref{thm:main_eta}.

\subsection{Convergence of the time-changed version of $S^{\ast}$}

In this section we show that $R$ is a time-changed version of $S^{\ast}$, and then we
show that  the scaling limit of $R$ is $W_\alpha$.

Put
\bel{eq:543}
\lambda(n) :=\inf\{k\in\mn : k - T(k)\geq n, \ S^\ast(k) > 0\},\quad n\in\mn_0
\ee
and note that, for $n\in\mn$, $\lambda(n)<\infty$ a.s.\ because $\zeta_1$, $\zeta_2,\ldots$ are a.s.\ positive.
\begin{lemma}\label{lemma:timeChange}
With probability $1$
\be\label{l1} S^\ast(\lambda(n))= R(n),\quad n\in\mn_0, \ee
that is, the sequence $R$ is obtained from $S^\ast$ by the time-change.
\end{lemma}
\begin{proof}
Since $T$ increases by unit jumps only, the definition of $\lambda$ ensures that
\be \label{lemma:timeChange:lambdaTLambda}
\lambda(n) - T(\lambda(n)) = n,\quad n\in\mn_0,
\ee
whence
\be \label{lemma:timeChange:sXi} S_\xi(k - T(k))|_{k = \lambda(n)} = S_\xi(n),\quad n\in\mn_0. \ee Thus, it remains to check that
\be\label{lemma:timeChange:sEta}
T \circ \lambda(n)=\bar \nu_\zeta \circ m(n),\quad n\in\mn_0\quad\text{a.s.}
\ee
Fix $n\in\mn$ and consider the events $A_n:=\{S^\ast(\lambda(n)-1)\leq 0\}$ and $A_n^c=\{S^\ast(\lambda(n)-1)>0\}$. We shall show that

\noindent (I) on $A_n$: \eqref{lemma:timeChange:sEta} holds;

\noindent  (II) on $A_n^c$: $T \circ \lambda(n) - T \circ \lambda(n-1) = 0$; 

\noindent (III) on $A_n^c$: $\bar \nu_\zeta \circ m(n) \leq T \circ \lambda(n)$.

These will guarantee that \eqref{lemma:timeChange:sEta} holds by an induction in $n$. Indeed,
$T \circ \lambda(0)=\bar \nu_\zeta \circ m(0)=1$, that is, \eqref{lemma:timeChange:sEta} holds for $n=0$. Assume that \eqref{lemma:timeChange:sEta} holds for $n=k-1$. The validity of \eqref{lemma:timeChange:sEta} a.s.\ on $A_k$ follows directly from (I). To prove that \eqref{lemma:timeChange:sEta} holds true a.s.\ on $A_k^c$, use (II), (III) and the induction assumption. These yield \[\bar \nu_\zeta \circ m(k) \leq  T \circ \lambda(k) = T \circ \lambda(k - 1) = \bar \nu_\zeta \circ m(k - 1),\] and the claim follows, for
$\bar \nu_\zeta \circ m$ is a.s.\ nondecreasing.

Before going further we state as the claims two properties of the model.
\begin{claim}\label{lemma:timeChange:claimTGrows}
For $k\in\mn_0$,
\[\{T(k+1)>T(k)\}=\{S^\ast(k)\leq 0\}.\]
\end{claim}

This is obvious, no proof is needed.
\begin{claim}\label{lemma:timeChange:claimMinimumAchieved}
For $k\in \mbN$ such that $S^\ast(k) \leq 0$, put
\[
\begin{split}
g(k) := \inf\{ l \in [2, k]:\ &S^\ast(l) \leq 0, \ S^\ast(l-1) > 0 \},
  \\ d(k) := \inf\{ r \geq k:\quad \: \: \: &S^\ast(r) \leq 0, \ S^\ast(r+1) > 0 \}.
\end{split}
\]
Then, for $k\in\mbN$,
\be\label{lemma:timeChange:minimumAchieved} \event{S^\ast(k)\leq 0} \subset \event{-S_\xi(i - T(i)) = m(i - T(i)) \ \text{{\rm for}} \ i\in[g(k) , d(k) + 1]}.
\ee
\end{claim}
\begin{proof}
Put $l=g(k)$ and $r=d(k)$. As $S^\ast(l-1) > 0$ it follows that $T(l) = T(l-1)$. Using representation \eqref{eq:repr} we infer \[S_\xi(l - T(l)) \leq - \mathcal{S}_\zeta(T(l))\quad \text{and}\quad S_\xi(l - 1 - T(l)) > - \mathcal{S}_\zeta(T(l)).\]
The fact that $\mathcal{S}_\zeta$ is a.s.\ nondecreasing implies that the minimum of $S_\xi$ on the interval $[0,\, l - T(l)]$ is achieved at $l - T(l)$. Finally, observe that the function $i \mapsto i-T(i)$ is constant on $[l, r+1]$, because $T(i + 1) = T(i) + 1$ for $i \in [l, r]$.
\end{proof}

With the claims at hand we now prove (I), (II) and (III).

\noindent {\sc Proof of (I).} Fix $\omega \in A_n$ and consider the number of elements of the sequence $(\zeta_k)_{k\in\mn}$ used in the construction of $(S^\ast(j))_{1\leq j \leq \lambda(n)}$. In view of \eqref{eq:repr} this number is $T\circ\lambda(n)$. Further, note that $S^\ast \circ \lambda(n) > 0$, that is, in the notation of Claim \ref{lemma:timeChange:claimMinimumAchieved},
\[d(\lambda(n) - 1) = \lambda(n) - 1. \] Recalling \eqref{lemma:timeChange:lambdaTLambda} and using Claim \ref{lemma:timeChange:claimMinimumAchieved} with $k = \lambda(n) - 1$ we conclude that $$-S_\xi(\lambda(n) - T(\lambda(n))) = m(\lambda(n)-T(\lambda(n))) = m(n).$$ Since $S^\ast(\lambda(n)) > 0$ and $S^\ast(\lambda(n) - 1) \leq 0$, we infer
\[ \begin{split}
  \mathcal{S}_\zeta \circ T(\lambda(n)) &> -S_\xi(\lambda(n) - T(\lambda(n))) = m(n), \\
  \mathcal{S}_\zeta \circ T(\lambda(n) - 1) &\leq -S_\xi(\lambda(n)-1 - T(\lambda(n) -1)) = m(n).
\end{split}
\]
This implies that the number of elements of the sequence $(\zeta_k)_{k\in\mn}$ used in the construction of $(S^\ast(j))_{1\leq j \leq \lambda(n)}$ is equal to $\bar \nu_\zeta\circ m(n)$, because it is the minimal number which makes $\mathcal{S}_\zeta$ greater than $m(n)$.

\noindent {\sc Proof of (II).} Fix $\omega \in A_n^c$. It follows from Claim \ref{lemma:timeChange:claimTGrows} that $\lambda(n - 1) = \lambda(n) - 1$. Also, Claim \ref{lemma:timeChange:claimTGrows} guarantees that $T(\lambda(n)-1) = T\circ \lambda(n)$. Hence, $$T\circ \lambda(n-1)=T(\lambda(n)-1)=T\circ \lambda(n).$$

\noindent {\sc Proof of (III).} The subsequent argument works for both $A_n^c$ and $A_n$.
Since $S^\ast(\lambda(n))>0$ a.s. and according to \eqref{eq:repr},
\begin{equation*}
\begin{split}
\mathcal{S}_\zeta\circ T\circ \lambda(n) &=S^\ast \circ \lambda(n) - S_\xi(\lambda(n)-T\circ \lambda(n))\\
&> m(\lambda(n) - T\circ \lambda(n)) = m(n),
\end{split}
\end{equation*}
applying $\bar \nu_\zeta$ to both sides of the last inequality yields $$T\circ \lambda(n) \geq \bar \nu_\zeta\circ m(n).$$
The proof of Lemma \ref{lemma:timeChange} is complete.
\end{proof}

Recall that $S^\ast=S^\ast_v$ depends on the parameter $v$  (hence, so does $R=R_v$), that $S^\ast_v(0)\leq 0$ a.s.\ and $v^{-1/2}S_v^\ast(0)\toP 0$ as $v\to\infty$.
\begin{lemma}\label{lemma:RConvergence}
Under the assumptions of Theorem \ref{thm:main_eta},
\begin{equation}\label{eq:inter}
\Big(\frac{S_\xi(\lfloor vt\rfloor)}{\sigma v^{1/2}}, \frac{R_v(\lfloor vt\rfloor)}{\sigma v^{1/2}}\Big)_{t\geq 0}~\Rightarrow~
(W(t), W_\alpha(t))_{t\geq 0},\quad v\to\infty.
\end{equation}
\end{lemma}
\begin{proof}
{The process $\Big(\frac{\bar \nu_\zeta(\sigma v^{1/2}t)}{b(v)}\Big)_{t\geq 0}$ is the first-passage time process for $\Big(\frac{\mathcal{S}_\zeta(\lfloor b(v)t\rfloor)}{\sigma v^{1/2}}\Big)_{t\geq 0}$. A composition of these is $\Big(\frac{\mathcal{S}_\zeta(\bar \nu_\zeta(\sigma v^{1/2}t))}{\sigma v^{1/2}}\Big)_{t \geq 0}$. Recalling \eqref{eq:conv_W_AND_U}, an application of Proposition \ref{prop:HenryStraka} in combination with the continuous mapping theorem yields}
\begin{multline}\label{eq:joint}
\Big(\frac{S_\xi(\lfloor vt \rfloor)}{\sigma v^{1/2}},\frac{-\min_{s\in [0,\,t ]}
(S^\ast_v(0)+S_\xi(\lfloor vs \rfloor))}{\sigma v^{1/2}},\frac{\mathcal{S}_\zeta(\bar \nu_\zeta({\sigma}v^{1/2}t))}{
{\sigma}v^{1/2}}\Big)_{t \geq 0}\\
~\Rightarrow~ (W(t),M(t), U_\alpha \circ U_\alpha^{\leftarrow}(t))_{t\geq 0}
\end{multline}
in $D^3$ in the product $J_1$-topology, where, as before, $(W,M)$ and $U_\alpha$ are assumed independent, and $M(t)=-\min_{s\in [0,t]}\,W(s)$ for $t\geq 0$.

Our next step is to prove that, as $v\to\infty$,
\begin{equation}\label{eq:inter2}
\Big(\frac{S_\xi(\lfloor vt \rfloor)}{\sigma v^{1/2}},
 \frac{\mathcal{S}_\zeta(\bar \nu_\zeta(-\min_{s\in [0,\,t]}(S^\ast_v(0)+S_\xi(\lfloor vs\rfloor))))}{\sigma v^{1/2}}\Big)_{t \geq 0}~
\Rightarrow~ (W(t), U_\alpha\circ U^\leftarrow_\alpha\circ M(t))_{t \geq 0}
\end{equation}
in $D^2$ in the product $J_1$-topology. To this end, we intend to invoke Lemma \ref{lemma:composition} given next. Being of principal importance for the proof of Lemma \ref{lemma:RConvergence}, Lemma \ref{lemma:composition} should also be useful as far as other problems involving compositions are concerned, not necessarily related to the setting of the present paper. The proof of this lemma is postponed to the \hyperref[Appendix]{Appendix}.

For $f\in D$, denote by ${\rm Disc}(f):=\{a: \ f(a-) \neq f(a)\}$ the set of discontinuities of $f$.
\begin{lemma}\label{lemma:composition}
For $\ninNo$, let $x_n, \ y_n\in D$, $y_n$ be nondecreasing and $y_0$ continuous. Assume that $\lim_{\ninf}x_n=x_0$ and $\lim_{\ninf} y_n=y_0$ in the $J_1$-topology in $D$ and that if, for some $t\geq 0$, $y_0(t)\in {\rm Disc}(x_0)$, then $\#\{u\geq 0: y_0(u)= y_0(t)\}= 1$. Then
\begin{equation}\label{eq:aux}
\lim_{\ninf}x_n \circ y_n= x_0\circ y_0
\end{equation}
in the $J_1$-topology in $D$.
\end{lemma}

It is known, see, for instance, Lemma 11.17 in \cite{Schilling+Partzsch:2014}, that,
for any fixed $a\geq 0$,
\begin{equation}\label{eq:const}
\mbP\big\{\#\{u\geq 0: M(u) = a\} = 1\big\}= 1.
\end{equation}
Since $U_\alpha\circ U^\leftarrow_\alpha$ and $M$  are independent processes, and the set
${\rm Disc}(U_\alpha\circ U^\leftarrow_\alpha)$ of discontinuities of $U_\alpha\circ U^\leftarrow_\alpha$ is a.s.\ countable, we conclude with the help of \eqref{eq:const} that
\be \mbP \event{\#\{u: M(u) = a\} = 1 \ \text{for}\ a \in {\rm Disc}(U_\alpha\circ U^\leftarrow_\alpha)} = 1.
\ee
Finally, we note that, for each $v\geq 0$, the process $t\mapsto -\min_{s\in [0,\,t]}S_\xi(\lfloor vs\rfloor)$ is a.s.\ nondecreasing,
and the process $M$ is a.s.\ nondecreasing and continuous. Thus, we have checked that
Lemma \ref{lemma:composition} applies to the processes discussed above or rather their
versions whose existence is secured by Skorokhod's representation theorem. As a result, we obtain \eqref{eq:inter2} and thereupon \eqref{eq:inter} because the summation operation (with two summands) is continuous whenever one of the summands is a continuous function, see, for instance, Theorem 4.1 in \cite{Whitt:1980}.
\end{proof}

\subsection{Convergence of the scaled $S^\ast$}
 Note that the sequence $\lambda=\lambda_v$ defined in \eqref{eq:543} depends on $v$.

\begin{lemma}\label{lemma:lambda_id}
For all $t_0>0$,
\be\label{final}
 \sup_{t\in[0,\,t_0]} \Big|\frac{\lambda_v(\lfloor v t\rfloor)}{v} - t\Big|\toP 0,\quad v\to\infty.
 \ee
\end{lemma}
\begin{proof}
For each $v>0$, the sequence $(\lambda_v(n)-n)_{n\geq 0}$ is nondecreasing. Hence,
\[
\sup_{t\in[0,\,t_0]} \Big|\frac{\lambda_v(\lfloor v t\rfloor)}{v} - t\Big|=
\Big|\frac{\lambda_v(\lfloor v t_0\rfloor)}{v} - t_0\Big|,
\]
and it suffices to prove that
\be\label{final1}
   \frac{\lambda_v(\lfloor v t_0\rfloor)}{v}  \toP t_0,\quad v\to\infty
 \ee
for all $t_0>0$.

It follows from \eqref{lemma:timeChange:lambdaTLambda} that
\[  \frac{\lambda_v(n)}{n}=\left(1-\frac{T_v(\lambda_v(n))}{\lambda_v(n)}\right)^{-1},\quad n\in\mn.\]
Since $\lambda_v(n)\geq n$ a.s., it is enough to check that, for any $\ve>0$,
\begin{equation}\label{eq:slln}
\sup_{n\geq \ve v}\frac{T_v(n)}{n}\toP 0, \quad v\to\infty.
\end{equation}
Observe that, for $\delta > 0$,
\be\label{lemma:lambda_id:eventZeta}
  \begin{split}
     \event{T_v(n) \leq  n \delta  }
      \supset \event{S_v^\ast(0)-|\xi_1|-\ldots-|\xi_{n - \lfloor n \delta\rfloor}|+ \zeta_1+ \ldots +\zeta_{\lfloor \delta n\rfloor } > 0}
   \\
\supset \event{S^\ast_v(0)-|\xi_1|-\ldots-|\xi_n|+ \zeta_1+ \ldots + \zeta_{\lfloor \delta n\rfloor} > 0}.
\end{split}
  \ee
By the strong law of large numbers for random walks and \eqref{eq:492},
\[\lim_{n\to\infty}\frac{-|\xi_1|-\ldots-|\xi_n| + \zeta_1+ \ldots +
  \zeta_{\lfloor \delta n\rfloor}}{n}=+\infty \quad \text{a.s.}
\]
whence
\be\label{lemma:lambda_id:eventEta}
\begin{split}
\Pb \event{\sup_{n\geq \ve v}\frac{T_v(n)}{n}\leq \delta} \geq & \\
& \Pb\event{S_v^\ast(0)+\inf_{n\geq \ve v}(-|\xi_1|-\ldots-|\xi_{n}| + \zeta_1+ \ldots + \zeta_{\lfloor \delta n\rfloor})  >0 }~ \to~ 1,
\end{split}
\ee
as $v\to\infty$ having utilized $v^{-1}S^\ast_v(0)\toP 0$ as $v\to\infty$.
\end{proof}

According to Lemma \ref{lemma:timeChange},
\[ R_v( {\lfloor vt\rfloor})= S_v^{\ast} \p{\lambda_v( {\lfloor vt\rfloor}
}= S_v^{\ast} \p[\Big]{\frac{\lambda_v( {\lfloor vt\rfloor}
)}{v}v},\quad t\geq 0,~v>0.\]
The time-change $t \mapsto v^{-1}\lambda_v (vt)$ is discontinuous and nondecreasing (rather than strictly increasing). Hence,
negligibility of the distance in $D$ between  $(v^{-1/2} S_v^\ast(\lfloor vt \rfloor))_{t\geq 0}$ and $(v^{-1/2} R_v( {\lfloor vt\rfloor}))_{t\geq 0}$ as $v\to\infty$ cannot be deduced from the definition of the $J_1$-topology. Lemma \ref{lemma:randomWalkTimeChangeConvergence} is designed to deal with this technicality. Its proof is deferred to the \hyperref[Appendix]{Appendix}.
\begin{lemma}\label{lemma:randomWalkTimeChangeConvergence}
For $n\in\mn_0$, let $\lambda_n$, $f_n\in D$, $\lambda_n$ be nonnegative and nondecreasing. Assume that, for all $T>0$,
\be\label{lemma:randomWalkTimeChangeConvergence:lambda}
\lim_{\ninf}\sup_{t \in [0,\, T]} |\lambda_n(t) - t| = 0
\ee
and
\be\label{lemma:randomWalkTimeChangeConvergence:fConv}
\lim_{n\to\infty} f_n\circ \lambda_n=f_0 
\ee
in the $J_1$-topology in $D$. For $n\in\mn$, denote by $(t_k^{(n)})_{k\in\mn}$ elements of the set ${\rm Disc}(\lambda_n)$ and, for $k\in\mn$, put $u_k^{(n)}:= \lambda_n(t_k^{(n)}-)$ and $v_k^{(n)}:= \lambda_n(t_k^{(n)})$. If, in addition to \eqref{lemma:randomWalkTimeChangeConvergence:lambda} and \eqref{lemma:randomWalkTimeChangeConvergence:fConv}, for all $T>0$,
\be \label{lemma:randomWalkTimeChangeConvergence:condition}
\lim_{n\to\infty} \sup_{k\geq 1}\sup_{s \in [u_k^{(n)}, v_k^{(n)})\cap [0,\, T]} |f_n(s) - f_n(u_k^{(n)}-)|=0,
\ee
then $$\lim_{n\to\infty}f_n=f_0$$ in the $J_1$-topology in $D$.
\end{lemma}

For later needs, put
$$Y_n(t):= \frac{\max_{1\leq i\leq \lfloor nt\rfloor}\, |\xi_i|}{n^{1/2}},\quad n\in\mn,~t\geq 0.$$
The assumption $\E\xi^2<\infty$ ensures that
\be\label{maxXi}
Y_n ~\Rightarrow~{\bf 0},\quad n\to\infty
\ee
in $D$, where ${\bf 0}(t):=0$ for $t\geq 0$.

According to the Skorokhod representation theorem in conjunction with Lemmas \ref{lemma:RConvergence} and \ref{lemma:lambda_id}, for any sequence $(v_n)_{n\in\mn}$ which diverges to $+\infty$ as $n\to\infty$, there exists a probability space which accommodates random elements $((S^{(n)}_\xi,  R_{v_n}^{(n)},  \lambda^{(n)}_{v_n},
 \hat Y_n))_{\ninN}$ satisfying, for $\ninN$,
\[
\left(\frac{S^{(n)}_\xi (\lfloor v_nt\rfloor)}{v_n^{1/2}}, \frac{R^{(n)}_{v_n} ({\lfloor v_nt\rfloor})}{v_n^{1/2}}, \frac{\lambda^{(n)}_{v_n}({\lfloor v_nt\rfloor})}{  v_n}, \hat Y_n \right)_{t\geq 0}\overset{{\rm d}}{=}
\left(
\frac{S_\xi( \lfloor v_nt\rfloor)}{ v_n^{1/2}}, \frac{ {R_{v_n}(\lfloor v_nt\rfloor)}}{v_n^{1/2}}, \frac{ {\lambda_{v_n}(\lfloor v_nt\rfloor)}}{v_n}, Y_n \right)_{t\geq 0}
\]
and
\bel{eq:1244}
\lim_{n\to\infty} \left(\frac{S^{(n)}_\xi( \lfloor v_nt\rfloor)}{\sigma v_n^{1/2}}, \frac{R^{(n)}_{v_n}(\lfloor v_nt\rfloor)}{\sigma v_n^{1/2}}, \frac{\lambda^{(n)}_{v_n}(\lfloor v_nt\rfloor)}{  v_n}, \hat Y_n \right)=
\left(
\hat W(t), \hat W_\alpha(t),t,0\right)\quad {{\rm a.s.\ in}}~ D^4,
\ee
where $(\hat W, \hat W_\alpha)$ is a copy of the process $(W, W_\alpha)$. Fix any $\omega$ such that \eqref{eq:1244} holds. The aforementioned new probability space also accommodates
copies $S^{\ast (n)}_{v_n}$ of $S^\ast_{v_n}$, for each $\ninN$.
Representation \eqref{l1} from Lemma \ref{lemma:timeChange} holds for these copies. We are going to apply Lemma \ref{lemma:randomWalkTimeChangeConvergence} with $$f_n(t):=\frac{R^{(n)}_{v_n}(\lfloor v_nt\rfloor )}{\sigma v_n^{1/2}} \ \text{and} \ \lambda_n(t):=\frac{\lambda^{(n)}_{v_n}(\lfloor v_nt\rfloor)}{v_n},\quad n\in\mn,~ t\geq 0.$$ For this particular choice, condition \eqref{lemma:randomWalkTimeChangeConvergence:condition} is justified by the convergence of the fourth coordinate in \eqref{eq:1244} and the fact that the supremum in \eqref{lemma:randomWalkTimeChangeConvergence:condition} does not exceed $\hat Y_n(T)$. By Lemma \ref{lemma:randomWalkTimeChangeConvergence},
\[
\lim_{n\to\infty} \frac{S^{\ast (n)}_{v_n} (\lfloor v_nt\rfloor )}{\sigma v_n^{1/2}}=\hat W_\alpha(t)\quad \text{{\rm in}}~D
\]
for the chosen $\omega$ and thereupon a.s. This completes the proof of Theorem \ref{thm:main_eta}.

\section{Properties of the limit process}\label{sec:properties}

In this section we discuss several properties of the limit process  $W^{(x)}_\alpha=(W_\alpha(x,t))_{t\geq 0}$
arising in Theorem \ref{thm:main} such as self-similarity, properties of excursions and a Markov property. We explain that $W^{(x)}_\alpha$ admits a representation as the solution to a stochastic equation with reflection. Alternatively, it can be thought of as a Feller Brownian motion on $[0,\infty)$ with a `jump-type' exit from $0$.

Let $\kappa>0$. We start by noting that the distribution tail of $\kappa\eta$ satisfies a counterpart of \eqref{eq:reg}, with $\ell$ replaced by $\kappa^{-\alpha}\ell$. Since the slowly varying function $\ell$ from \eqref{eq:reg} does not pop up in the limit process $W^{(x)}_\alpha$, limit relations \eqref{eq:impo} and \eqref{eq:impo1} remain
valid upon replacing $(\eta_n)_{n\in\mn}$ with $(\kappa \eta_n)_{n\in\mn}$. Further, observe that the distribution of $W^{(x)}_\alpha$ does not
change when replacing in \eqref{eq:representation} the process $U_\alpha$ with any other drift-free $\alpha$-stable subordinator (without killing) $V_\alpha$, say. Indeed, the distribution of $V_\alpha$ coincides with the distribution of $(U_\alpha(ct))_{t\geq 0}$ for some $c>0$. An inverse $\alpha$-stable subordinator $V_\alpha^\leftarrow$ has the same distribution as $(c^{-1} U_\alpha^\leftarrow(t))_{t\geq 0}$. Finally, the composition of $(U_\alpha(ct))_{t\geq 0}$ and $(c^{-1}U_\alpha^\leftarrow(t))_{t\geq 0}$ is $(U_\alpha\circ U_\alpha^\leftarrow(t))_{t\geq 0}$ (for all $\omega$).

For all $x\geq 0$, the processes $W_\alpha^{(x)}$ are homogeneous Markov processes. Furthermore, these are Feller processes, see Theorem 3.11 in Chapter II of \cite{Blumenthal:1992} and for
 $x\neq y$, $W_\alpha^{(x)}$ and $W_\alpha^{(y)}$ have the same transition probabilities.

It follows from the definition that $W_\alpha^{(x)}$ behaves like the Brownian motion $W$ until it hits $0$. Thus, $W_\alpha^{(x)}$ is a Feller Brownian motion, that is, a Markov extension of a Brownian motion after hitting $0$.
\begin{remk}
In the theory of Markov processes one usually considers a process $Y$, say under the collection of measures $\Pb_x(\cdot):=\Pb(\cdot|Y(0)=x)$ for $x\geq 0$.
For our needs it is more convenient to work with the collection of processes $W_\alpha^{(x)}$, indexed by the initial starting point $x\geq 0$,
under a single probability measure $\Pb$. We hope this does not lead to a confusion.
\end{remk}

By Theorem \ref{thm:main},
\begin{large}
\begin{center}
\begin{tikzcd}
(v^{-1/2} \tilde S(v ct))_{t\geq 0} \arrow[equal]{d}\quad\quad~\Rightarrow ~ & (W_\alpha(ct))_{t\geq 0} \\
(c^{1/2}(vc)^{-1/2}\tilde S(v ct))_{t\geq 0}~\Rightarrow~& (c^{1/2} W_\alpha(t))_{t\geq 0} \\
\end{tikzcd}
\end{center}
\end{large}
that is, the process $W_\alpha$ is self-similar with exponent $1/2$. Using this in combination with
$1/2$ self-similarity of a Brownian motion started at $x$ and stopped upon hitting $0$ and the strong Markov property of $W_\alpha^{(x)}$ we conclude that, for any $c>0$ and $x\geq 0$, the process $(W_\alpha^{(x)}(ct))_{t\geq 0}$ has the same
distribution as $\left(c^{1/2} W^{(c^{-1/2} x)}_\alpha(t)\right)_{ t\geq 0}$.

Now we describe the process $W_\alpha^{(x)}$ from the resolvent point of view. To this end, define the resolvent
\[U^\lambda f(x):=\E \int_0^\infty e^{-\lambda s}f(W^{(x)}_\alpha(s)){\rm d}s,\quad x\geq 0.\]
Denote by $V^\lambda f(x)$ the resolvent of a Brownian motion on $[0,\infty)$ killed at $0$. It is known that
\[
V^\lambda f(x)=\int_0^\infty v^\lambda(x,y) f(y){\rm d}y,\quad x>0,
\]
where $ v^\lambda(x,y):=\frac1{\sqrt{2\lambda}}(e^{-\sqrt{2\lambda}|x-y|}-e^{-\sqrt{2\lambda}|x+y|})$ for $x,y>0$, see p.~56 in \cite{Blumenthal:1992}. Invoking the general theory of Markov processes one can show that
\[U^\lambda f(x)= V^\lambda f(x) + \E_x e^{-\lambda \sigma_0} U^\lambda f(0)=V^\lambda f(x) +  e^{-x\sqrt{2\lambda}} U^\lambda f(0),\quad x>0,\]
where $\sigma_0$ is the first hitting time of $0$, see p.~57 in \cite{Blumenthal:1992}. Note that this formula holds true for any Markov extension of a Brownian motion after hitting $0$. It follows from Theorem 3.11 in Chapter II of \cite{Blumenthal:1992} that
\be\label{eq:resolvent}
U^\lambda f(0)=\Delta_\lambda^{-1} \int_0^\infty V^\lambda f(x) \frac{{\rm d}x}{x^{1+\alpha}},
\ee
where
\[\Delta_\lambda= \int_0^\infty   (1-e^{-x\sqrt{2\lambda}}) \frac{{\rm d}x}{x^{1+\alpha}}=\frac{(2\lambda)^{\alpha/2}\Gamma(1-\alpha)}{\alpha}.\]
The last equality is obtained with the help of integration by parts.
\begin{remk}
The book \cite{Blumenthal:1992} only focuses on the case $\lambda=1$. However, the case $\lambda\neq 1$ is analogous. Note that the value of the norming constant
$\Delta_\lambda$ can be derived from the equality $U^\lambda 1(x)=\lambda^{-1}$, $x\geq 0$.
\end{remk}
\begin{remk}
Equation \eqref{eq:resolvent} entails
that the entrance law for $W^{(x)}_\alpha$ is given by
$$\frac{\alpha}{ 2^{\alpha/2}\Gamma(1-\alpha)}\int_0^\infty P_t^0(x, {\rm d}y) \frac{{\rm d}x}{x^{1+\alpha}},$$ see Chapter V, \S2 in \cite{Blumenthal:1992},
where
$$P_t^0(x, {\rm d}y)=(2\pi t)^{-1/2} (e^{-\frac{(x-y)^2}{2t}}-e^{-\frac{(x+y)^2}{2t}}){\rm d}y$$
 is the transition kernel of the semigroup for a Brownian motion killed at $0$.
\end{remk}
Summarizing, we conclude that the resolvent kernel $r^\lambda(x,y)$ of $W_\alpha$ is given by
\[
r^\lambda(x,y)= v^\lambda(x,y)+\Delta_\lambda^{-1} \int_0^\infty v^\lambda(z,y)
 \frac{{\rm d} z}{z^{1+\alpha}},\quad x,y>0.
\]

Now we are going to point out the distributions of $(W(t), -\min_{s\in [0,\,t]} W(s))$ and $U_\alpha \circ U_\alpha^{\leftarrow}(t)$. According to Problem 1 on p.~27 in \cite{Ito+McKean:1996},
\[
\Pb\{W(t)\in {\rm d}a, \max_{s\in [0,\,t]} W(s)\in {\rm d}b\}=\left(\frac{2}{\pi t^3}\right)^{\frac12}
(2b-a)e^{(2b-a)^2/2t} {\rm d}a {\rm d}b,\quad t>0, \ 0\leq b,\ b\geq a.
\]
As a consequence,
\[
\Pb\{W(t)\in {\rm d}a,\, -\min_{s\in [0,\,t]} W(s)\in {\rm d}b\}=\left(\frac{2}{\pi t^3}\right)^{\frac12}(2b+a)e^{(2b+a)^2/2t}{\rm d}a {\rm d}b,\quad t>0, \ 0\leq b,\ b+a\geq 0.
\]
Notice that $U_\alpha (U_\alpha^\leftarrow (t))-t$ is the overshoot of the process $U_\alpha$ at $t>0$. It follows from the Dynkin-Lamperti asymptotics (see, for instance, p.~135 in \cite{Kyprianou:2014}) and self-similarity of $U_\alpha\circ U_\alpha^{\leftarrow} $ with exponent $1$
(which is a consequence of \eqref{eq:joint} restricted to the third coordinate) that $$\Pb\Big\{\frac{U_\alpha \circ U_\alpha^{\leftarrow}(t)}{t}\in{\rm d}x\Big\}=
\frac{\sin(\pi\alpha)}{\pi}\frac{\1_{(1,\infty)}(x)}{(x-1)^\alpha x}{\rm d}x, \quad x>0,
$$
whence
 \[
 \Pb\{U_\alpha \circ U_\alpha^{\leftarrow}(t)\in{\rm d}x\}=
\frac{t^\alpha\sin(\pi\alpha)}{\pi}\frac{\1_{(t,\infty)}(x)}{(x-t)^\alpha x}{\rm d}x,\quad x>0.
 \]
Absolute continuity of this distribution particularly implies that, for all $s>0$, \[\Pb\{W^{(x)}_\alpha(s)=0\}=0.\]
Thus, the process $W^{(x)}_\alpha$ spends zero time at $0$ with probability $1$.

Even though the distributions of $(W(t), -\min_{s\in [0,\,t]} W(s))$ and $U_\alpha \circ U_\alpha^{\leftarrow}(t)$ are known explicitly we have been unable to find an explicit form of the transition density of $W^{(x)}_\alpha$.

According to \cite{Pilipenko:2012}, there exists a unique pair of nonnegative  processes $(\hat W^{(x)}_\alpha, L^{(x)}_\alpha)$ satisfying
a generalized Skorokhod reflection problem
\bel{eq:ReflW}
\hat W^{(x)}_{\alpha}(t)=x+ W(t)+U_\alpha(L_\alpha^{(x)}(t)),\quad t\geq 0.
\ee
Here, the unknown process $L_\alpha^{(x)}$ is a.s.\ continuous, nondecreasing and satisfies
\bel{eq:L}
L_\alpha^{(x)}(0)=0 \ \text{and} \ \int_{[0,\,\infty)}\1_{\{\hat W^{(x)}_\alpha(s)>0\}}{\rm d}L_\alpha^{(x)}(s)=0.
\ee
Comparing \eqref{eq:ReflW} and \eqref{eq:representation} we conclude that
$\hat W^{(x)}_{\alpha}(t)= W^{(x)}_{\alpha}(t)$ and  $L_\alpha^{(x)}(t)=U_\alpha^{\leftarrow} \circ ((-x+M(t))^+)$ for $t\geq 0$.

It follows from \eqref{eq:ReflW} and \eqref{eq:L} (or just from formula  \eqref{eq:representation}, or the It\^{o} excursion theory together with \eqref{eq:resolvent}) that
the increments of $W^{(x)}_\alpha$ coincide with those of $W$ while $W^{(x)}_\alpha$ is positive. If $W^{(x)}_\alpha$ is discontinuous at $t$, then $W_\alpha(x,t-)=0$ and $W_\alpha(x,t)=U_\alpha(L_\alpha^{(x)}(t))- U_\alpha(L_\alpha^{(x)}(t)-)$, that is, jumps from $0$ are governed by the process $U_\alpha$ and further controlled by an ``inner'' time given by $L_\alpha^{(x)}$. Further, if $(l(t_0), r(t_0))$ is an excursion interval of $W^{(x)}_\alpha$ that straddles a point  $t_0>0$, then $W_\alpha(x, l(t_0))>0$ a.s. This implies that there is a `jump-type' exit from $0$ rather than a `continuous' exit.
\begin{remk}
An alert reader will notice that any right neighborhood $(r(t_0),r(t_0)+\ve)$ contains an infinite number
of excursions with probability $1$. Thus, the picture is similar to the behavior of the excursions of a Brownian motion.
\end{remk}
The process $L_\alpha^{(x)}$ is a continuous additive functional of the Markov process $W^{(x)}_\alpha$ whose points of increase
are supported by the set $\{s\geq 0\ : \ W^{(x)}_\alpha(s)=0\}$, see p.~68--69 in \cite{Blumenthal:1992}.
Thus, $L_\alpha^{(x)}$ is the {\it Blumenthal-Getoor local time} up to a multiplicative constant.
In particular, the process $L_\alpha^{(x)}$ is $\cF_t$-adapted, where
$\cF_t$ is a completion by sets of zero measure of the $\sigma$-algebra generated
by $(W_\alpha^{(x)}(s))_{s\in[0,\,t]}$. This claim, which is not obvious, follows, for instance, from either of the following two representations for $L_\alpha^{(x)}$. The first one, in Theorem \ref{thm:local1}, is in terms of the number of jumps of $W^{(x)}_\alpha$.
The other, in Theorem \ref{thm:local2}, is in terms of the number of large excursions of
 $W^{(x)}_\alpha$ up to time $t$.
\begin{theorem}\label{thm:local1}
For any $T>0$, the convergence
\begin{multline*}
\lim_{\ve\to0+}\ve^\alpha \big(\mbox{the number of jumps of } W^{(x)}_\alpha~ \mbox{on } [0,\,t]~ \mbox{which are not smaller than}~
\ve\big)\\=
\lim_{\ve\to 0+} \ve^\alpha \sum_{s\in[0,\,t]}\1_{\{W^{(x)}_\alpha(s)-W^{(x)}_\alpha(s-)\geq \ve\}}=
L_\alpha^{(x)}(t)
\end{multline*}
is uniform in $t\in [0,\,T] $ with probability $1$.
\end{theorem}
\begin{theorem}\label{thm:local2}
For any $T>0$, the convergence
\begin{multline*}
\lim_{\ve\to0+}\ve^{\alpha/2} \big(\mbox{the number of excursion intervals  of  } W^{(x)}_\alpha~ \mbox{on }~ [0,\,t]~\\
\mbox{whose lengths are not smaller than } \ve\big)
=\frac{\Gamma((1-\alpha)/2)}{(\pi 2^\alpha)^{1/2}}\varepsilon^{-\alpha/2}L_\alpha^{(x)}(t)
\end{multline*}
is uniform in $t\in [0,\,T] $ with probability $1$.
\end{theorem}
\begin{proof}[Proof of Theorem \ref{thm:local1}]
The proof of this result can be found in \cite{Blumenthal:1992}. We recall its main steps because similar arguments are used in the proof of Theorem \ref{thm:local2}, and also for completeness.

Consider the L\'{e}vy-It\^{o} representation of $U_\alpha$
\[
U_{\alpha}(t)= \int_{s\in[0,\,t]}\int_{[0,\,\infty)} u N({\rm d}s, {\rm d}u),\quad t\geq 0.
\]
Here, $N:=\sum_k \delta_{(t_k, u_k)}$ is a Poisson random measure on $[0,\infty)\times (0,\infty]$ with
intensity measure ${\rm LEB}\otimes \nu$; $\delta_{(t,x)}$ is the probability measure concentrated at $(t,x)$; ${\rm LEB}$ is the Lebesgue
measure on $[0,\infty)$, and $\nu$ is the L\'{e}vy measure given by
\begin{equation}\label{eq:levy}
\nu({\rm d}u)=\alpha u^{-1-\alpha}\1_{(0,\infty)}(u){\rm d}u,\quad u\in\mathbb{R}.
\end{equation}

For any $\ve>0$,
\begin{multline}\label{eq:559}
\mbox{the number of jumps of}~ W^{(x)}_\alpha~ \mbox{on } [0,\,t]~ \mbox{ which are not smaller than}~ \ve\\=
\mbox{the number of jumps of }~ U_\alpha~ \mbox{ on } [0,\,L_\alpha^{(x)}(t)]~
\mbox{which are not smaller than}~ \ve\\= N([0,\,L_\alpha^{(x)}(t)]\times [\ve,\infty)).
\end{multline}
By the strong law of large numbers for Poisson processes, for any fixed $t\geq 0$,
\bel{eq:lim_pois}
\lim_{\ve\to 0+}\ve^\alpha N([0,\,t]\times [\ve,\infty))=t\quad \text{a.s.}
\ee
because, for each $t>0$, the process $(N([0,\,t]\times [u^{-1},\infty)))_{u>0}$ is an
inhomogeneous Poisson
process of intensity $u\mapsto u^{\alpha}t$. As a consequence, relation \eqref{eq:lim_pois} holds
true with probability $1$ for all rational $t\geq 0$.
Since, for each $\ve>0$, the process $(\ve^\alpha N([0,\,t] \times [\ve,\infty)))_{t\geq 0}$ is a.s.
\ nondecreasing, and the limit function in \eqref{eq:lim_pois} is continuous, we infer,
for all $T>0$,
\[
\lim_{\ve\to0+}\sup_{t\in [0,\,T]}\big|\ve^\alpha N([0,\,t]\times [\ve,\infty))- t\big|=0\quad\text{a.s.}
\]
This in combination with \eqref{eq:559} and a.s.\ continuity of $L_\alpha^{(x)}$ completes the proof.
\end{proof}

\begin{proof}[Proof of Theorem \ref{thm:local2}]
Let $\theta_1$, $\theta_2,\ldots$ be independent copies of $\theta:=\inf\{ t\geq 0\ :\ 1+W(t)=0\}$.
By self-similarity of $W$ and the fact that $-W$ has the same distribution as $W$,
\[
u^2\theta ~\overset{{\rm d}}=~ \inf\{ t\geq 0\ :\ u+W(t)=0\}~\overset{{\rm d}}=~ \inf\{ t\geq 0\ :\ -u+W(t)=0\},\quad u> 0,
\]
where $\overset{{\rm d}}=$ denotes equality of distributions. Using this in combination with formula 2) on p.~25 in \cite{Ito+McKean:1996} we conclude that
$$\Pb\{u^2 \theta\in {\rm d}z\}=\frac{u}{\sqrt{2\pi z^3}} e^{-u^2/2z}\1_{(0,\infty)}(z){\rm d}z=:f(u,z){\rm d}z,\quad z\in\R,~u>0.$$
We shall use the Poisson random measure $N$ (or rather its atoms) defined in the proof of Theorem \ref{thm:local1}. Recall that $N$ and $W$ are independent. We proceed by noting that
\begin{multline*}
\big(\text{the number of excursion intervals of}~~ W^{(x)}_\alpha~~ \text{starting in}~ [0,\,t]~~\\ \text{whose lengths are not smaller than}~ \ve \big)_{t\geq 0}~\overset{{\rm d}}=~\Big(\sum_{t_k\leq L_\alpha^{(x)}(t) }\sum_{u_k} \1_{\{u_k^2\theta_k
\geq \ve\}}\Big)_{t\geq 0}.
\end{multline*}
Further,
$$
\left(
\sum_{t_k\leq t}\sum_{u_k} \1_{\{u_k^2\theta_k \geq \ve\}}\right)_{t\geq 0}~\overset{{\rm d}}
=~ \left(
\int_{[0,\,t]} \int_{[\varepsilon,\infty)} M({\rm d}s, {\rm d}v)\right)_{t\geq 0},
$$
where $M$ is a Poisson random measure on $[0,\infty)\times (0,\infty]$ with
intensity measure ${\rm LEB}\otimes \rho$, and $\rho$ is a measure on $(0,\infty)$ defined by $$\rho({\rm d}z)=\int_{(0,\infty)}f(u,z)\nu({\rm d}u){\rm d}z,\quad z\in\R$$ with the L\'{e}vy measure $\nu$ defined in \eqref{eq:levy}.
In particular,
\begin{multline*}
\nu([\ve, \infty))=\int_\varepsilon^\infty \int_0^\infty \frac{u}{\sqrt{2\pi z^3}} e^{-u^2/2z}  \frac{\alpha }{u^{1+\alpha}}{\rm d}u{\rm d}z=
\frac{\alpha}{2(\pi 2^\alpha)^{1/2}}\int_{\ve}^\infty z^{-1-\frac{\alpha}{2}}{\rm d}z \int_0^\infty e^{-s} s^{- \frac{1+\alpha}{2}}{\rm d}s\\=
\frac{\Gamma((1-\alpha)/2)}{(\pi 2^\alpha)^{1/2}}\varepsilon^{-\alpha/2},
\end{multline*}
where the second equality follows by the change of variable $s =u^2/(2z)$ .

The remaining part of the proof, which is similar to the corresponding part of the proof of Theorem \ref{thm:local1},
commences with checking the asymptotic relation: for any fixed $t\geq 0$,
$$\lim_{\ve\to 0+}\ve^{\alpha/2} M([0,\,t] \times [\ve,\infty))=\frac{\Gamma((1-\alpha)/2)}{(\pi 2^\alpha)^{1/2}}t\quad \text{a.s.}$$
Observe that the number of excursion intervals of $W^{(x)}_\alpha$
starting in $[0,\,t]$ whose lengths are not smaller than $\ve$ exceeds at most by one
the number of excursion intervals of $W^{(x)}_\alpha$ belonging to $[0,\,t]$ whose lengths are not smaller than $\ve$.
\end{proof}

\begin{appendix}
\section*{Appendix}\label{Appendix}

In this section we collect a couple of technical results related to the $J_1$-convergence.
  We start with a proposition which follows from the definition of the $J_1$-convergence in $D$. The result is used for the justification of Proposition \ref{corl:excursion_parts}.
\begin{proposition}\label{lemma:excursion_parts}
For $n\in\mn_0$, let $f_n\in D$. Assume that for a sequence $(T_n)_{n\in\mn_0}$ of nonnegative
numbers the following limit relations hold:
\begin{itemize}

\item $\lim_{n\to\infty}T_n=T_0$;

\item $\lim_{n\to\infty}f_n(\cdot\wedge T_n)=f_0(\cdot\wedge T_0)$ in $D$;

\item $\lim_{n\to\infty} f_n(\cdot+ T_n)= f_0(\cdot+ T_0)$ in $D$.
\end{itemize}
Then $\lim_{n\to\infty} f_n=f_0$ in $D$.
\end{proposition}

Proposition \ref{prop:HenryStraka}, borrowed from Proposition 2.3 in \cite{Straka+Henry:2011}, is used in the proof of Lemma \ref{lemma:RConvergence}. We write $D([0,\infty)\times\mbR^d)$ for the Skorokhod space of c\`{a}dl\`{a}g functions defined on $[0,\infty)\times\mbR^d$.
\begin{proposition}\label{prop:HenryStraka}
For $n\in\mn_0$, let $(\alpha_n, \beta_n)\in D([0,\infty)\times\mbR^d)$. Assume that, for $n\in\mn$, $\alpha_n$ are nondecreasing, nonnegative and unbounded, that $\alpha_0$ is increasing and unbounded, and that $\lim_{n\to\infty}(\alpha_n,\beta_n)=(\alpha_0,\beta_0)$ in the $J_1$-topology in $D([0,\infty)\times\mbR^d)$. Then $\lim_{n\to\infty} \beta_n\circ \alpha_n^\leftarrow= \beta_0\circ \alpha_0^\leftarrow$ in the $J_1$-topology in $D(\mbR^d)$, where, for $n\in\mn_0$ and $t\geq 0$, $\alpha_n^\leftarrow(t):=\inf\{s\geq 0\ : \ \alpha_n(s)>t\}$.
\end{proposition}

The function $\alpha_n^\leftarrow$ is called generalized inverse of $\alpha_n$ or the first-passage time function of $\alpha_n$.

We  proceed with a classical characterization of the $J_1$-convergence which can be found in
Proposition 6.5 of \cite{Ethier+Kurtz:1986}.
\begin{proposition}\label{prop:criterion}
For $\ninNo$, let $z_n\in D$. Then $\lim_{\ninf} z_n = z_0$ in the $J_1$-topology
in $D$
if, and only if, for any $u_0\geq 0$ and any sequence $(u_n)_{\ninN}$ of nonnegative numbers satisfying $\lim_{\ninf}u_n = u_0$, the following conditions hold.
\begin{enumerate}[label=\textbf{C.\roman*}]
\item \label{prop:criterion:1} All limit points of $(z_n(u_n))_{\ninN}$ are either $z_0(u_0)$ or $z_0(u_0-)$.
\item \label{prop:criterion:2} If $\lim_{\ninf}z_n(u_n)=z_0(u_0)$, then
$\lim_{\ninf}z_n(v_n)=z_0(u_0)$ for any sequence $(v_n)_{\ninN}$ satisfying
$v_n \geq u_n$ for $\ninN$ and $\lim_{\ninf}v_n=u_0$.

\item \label{prop:criterion:3} If $\lim_{\ninf}z_n(u_n)=z_0(u_0-)$, then $\lim_{\ninf}z_n(v_n)=z_0(u_0-)$ for any sequence $(v_n)_{\ninN}$ satisfying
$v_n \leq u_n$ for $\ninN$ and $\lim_{\ninf}v_n=u_0$.
\end{enumerate}
\end{proposition}

Proposition  \ref{prop:criterion} will now be essentially used for the proofs
of Lemmas \ref{lemma:composition} and \ref{lemma:randomWalkTimeChangeConvergence}.
\begin{proof}[Proof of Lemma \ref{lemma:composition}]
Fix any $t_0>0$ and let $(t_n)_{\ninN}$ be a sequence satisfying $\lim_{n\to\infty}t_n=t_0$. Since $y_0$ is continuous by assumption, the $J_1$-convergence $\lim_{\ninf} y_n=y_0$ is equivalent to locally uniform convergence. This entails $\lim_{\ninf} y_n(t_n)=y_0(t_0)$.

To prove \eqref{eq:aux} we intend to show that Conditions \textbf{C.i,ii,iii} of Proposition
 \ref{prop:criterion} hold with
$z_n=x_n\circ y_n$, $u_n=t_n$ and $u_0=t_0$. While doing so, we use the other implication of
Proposition \ref{prop:criterion} with $z_n=x_n$, $u_n=y_n(t_n)$ and $u_0=y_0(t_0)$, namely,
the passage from the $J_1$-convergence $\lim_{\ninf}x_n=x_0$ to the corresponding Conditions \textbf{C.i,ii,iii}.

Condition  \ref{prop:criterion:1}. In view of $\lim_{\ninf}x_n= x_0$, Condition \textbf{C.i} of Proposition \ref{prop:criterion} tells us that the limit points of the sequence $(x_n\circ y_n(t_n))_{\ninN}=(x_n(u_n))_{\ninN}$ are either $x_0(u_0)=x_0\circ y_0(t_0)$ or $x_0(u_0-)=x_0(y_0(t_0)-)$. Thus, it suffices to prove that either $x_0(y(t_0)-) = x_0\circ y_0(t_0)$ or $x_0(y_0(t_0)-)=x_0\circ y_0(t_0-)$.  Indeed, if $y_0(t_0) \notin {\rm Disc}(x_0)$, then $x_0(y_0(t_0)-)=x_0\circ y_0(t_0)$. If $y_0(t_0)\in {\rm Disc}(x_0)$, then using the assumptions that  $y_0$ is nondecreasing and that $\#\{u\geq 0: y_0(u)= y_0(t)\}= 1$ we infer $y_0(s)<y_0(t_0)$ for any $s<t_0$, whence $x_0(y_0(t_0)-) = x_0\circ y_0(t_0-)$. It remains to note that, in view of right-continuity, Condition  \ref{prop:criterion:1} obviously holds true for $t_0=0$.

Condition  \ref{prop:criterion:2}. Assume that $\lim_{\ninf} t_n=t_0$ and $\lim_{\ninf}x_n \circ y_n(t_n)= x_0\circ y_0(t_0)$. Let $(s_n)_{\ninN}$ be any sequence satisfying $s_n \geq t_n$ for $\ninN$ and $\lim_{\ninf}s_n=t_0$.
 Since $y_n$ is nondecreasing, we infer $$v_n=y_n(s_n)\geq y_n(t_n)=u_n,\quad \ninN.$$
It has already been mentioned that $\lim_{\ninf}y_n(s_n)=y_0(t_0)$ in view of continuity of $y_0$. Thus,
using $\lim_{\ninf}x_n= x_0$ and invoking Condition \textbf{C.ii} of Proposition \ref{prop:criterion} we conclude that
\[x_n\circ y_n(s_n)=x_n (v_n)~\to~ x_0(u_0) = x_0\circ y_0(t_0),\quad \ninf.\]

Condition  \ref{prop:criterion:3} can be checked analogously.
\end{proof}

\begin{proof}[Proof of Lemma \ref{lemma:randomWalkTimeChangeConvergence}]
Fix any $T>0$. Let $(t_n)_{\ninNo}$ be any sequence satisfying $t_n\in [0, T)$ for $n\in\mn_0$. Assume that the limit $\lim_{n\to\infty}f_n(t_n)$ exists. We shall use Proposition \ref{prop:criterion}.

For a function $x$, denote by ${\rm Range}(x)$ its range, that is, the set of all possible values of $x$. To verify Condition \ref{prop:criterion:1} of Proposition \ref{prop:criterion} we have to show that $\lim_{n\to\infty}f_n(t_n) \in \{f(t_0-),f(t_0)\}$. Assume first that there exists a subsequence $(t_{n_k})_{k\in\mn}$ such that $t_{n_k}\in {\rm Range}(\lambda_{n_k})$ for all $k\in\mn$. Without loss of generality, we can and do assume that $t_n \in {\rm Range}(\lambda_{n})$ for all $n\in\mn$.

For $\ninN$, put \[\mu_n := \lambda_n^{\leftarrow}(t_n)\] ($\lambda_n^{\leftarrow}$ is right-continuous generalized inverse of $\lambda_n$)
and note that
  \[
  \lim_{n\to\infty}f_n(t_n)= \lim_{n\to\infty}f_n\circ\lambda_n(\mu_n).
  \]
In view of \eqref{lemma:randomWalkTimeChangeConvergence:lambda} it follows that $\lim_{n\to\infty} \mu_n = t_0$. Formula \eqref{lemma:randomWalkTimeChangeConvergence:fConv} and Condition  \ref{prop:criterion:1} of Proposition \ref{prop:criterion} imply that $\lim_{n\to\infty}f_n\circ\lambda_n(\mu_n)\in \{f(t_0-), f(t_0)\}$, whence
  \[
  \lim_{n\to\infty}f_n(t_n)\in \{f(t_0-), f(t_0)\}.
  \]

Assume now that $t_n \notin {\rm Range}(\lambda_{n})$ for all $n\in\mn$ (we do not need to investigate an intermediate situation in which $t_n \notin {\rm Range}(\lambda_{n})$ for some $n$ and $t_n\in {\rm Range}(\lambda_{n})$ for the other $n$; indeed, passing to a subsequence we can ensure that exactly one of these alternatives prevails for all values of indices). Then, with the same $\mu_n$ as before,
\be\label{lemma:randomWalkTimeChangeConvergence:uNvN}
u_n:=\lambda_n(\mu_n-) \leq t_n < \lambda_n(\mu_n)=: v_n,\quad n\in\mn.
\ee
For each fixed $n$, there are two possibilities: either $\lambda_n(\mu)<u_n$ for $\mu < \mu_n$ or $\lambda_n(\mu)=u_n$ for $\mu\in [\mu_n - \ve_n, \mu_n]$ for some $\ve_n>0$. Assuming that the first possibility prevails for all $n\in\mn$, we select a sequence $(\rho_n)_{n\in\mn}$ satisfying
$\rho_n < \mu_n$ for all $n\in\mn$, $\lim_{n\to\infty} \rho_n = t_0$ and
\[\lim_{n\to\infty}|f_n(\lambda_n(\rho_n)) - f_n(u_n-)|=0, \quad \lim_{n\to\infty}|\lambda_n(\rho_n) - u_n|=0.\]
This is possible because 
\[\lim_{\mu \to \mu_n-}f_n(\lambda_n(\mu))=\lim_{t \to u_n-} f_n(t)= f_n(u_n-).
  \]
Using \eqref{lemma:randomWalkTimeChangeConvergence:uNvN} in combination with \eqref{lemma:randomWalkTimeChangeConvergence:condition} yields \[|f_n(t_n)-f_n(\lambda_n(\rho_n))|\leq |f_n(t_n)-f_n(u_n-)|+|f_n(u_n-)-f_n(\lambda_n(\rho_n))| \to 0,\quad n\to\infty.\]
Hence, $\lim_{n\to\infty} f_n(\lambda_n(\rho_n))$ exists and is equal to 
$\lim_{n\to\infty} f_n(t_n)$. According to \eqref{lemma:randomWalkTimeChangeConvergence:fConv} and Condition \ref{prop:criterion:1} of Proposition \ref{prop:criterion} we have $\lim_{n\to\infty} f_n(t_n) \in  \{f(t_0-), f(t_0)\}$.

Assume now that, for each $n$, there exists $\ve_n>0$ such that $\lambda_n(\mu)=u_n$ for $\mu \in [\mu_n-\ve_n, \mu_n]$. Let $(\rho_n)_{n\in\mn}$ be any sequence satisfying $\rho_n \in [\mu_n-\ve_n, \mu_n)$ for $n\in\mn$ and $\lim_{n\to\infty} \rho_n = t_0$. As a consequence of $t_n\in[u_n, v_n)$ (see \eqref{lemma:randomWalkTimeChangeConvergence:uNvN}), as $n\to\infty$,
\[|f_n(t_n)-f_n(\lambda_n(\rho_n))|= |f_n(t_n)-f_n(u_n)|\leq
|f_n(t_n)-f_n(u_n-)|+|f_n(u_n-)-f_n(u_n))| \to 0.
  \]
Hence, $\lim_{n\to\infty} f_n(\lambda_n(\rho_n))$ exists and is equal to
$\lim_{n\to\infty} f_n(t_n)$. By the same argument as before we infer $\lim_{n\to\infty} f_n(t_n)\in  \{f(t_0-), f(t_0)\}$.

Conditions \ref{prop:criterion:2} and \ref{prop:criterion:3} of Proposition \ref{prop:criterion} can be verified similarly.
\end{proof}
\end{appendix}
%
%

\begin{acks}[Acknowledgments]
We thank two anonymous referees for many useful suggestions which significantly improved the presentation of our results. Our special thanks go to one of the referees who has kindly informed us about the line of research on the oscillating random walks and provided a list of relevant references.
\end{acks}
\begin{funding}
A. Iksanov and A. Pilipenko acknowledge support by the National Research Foundation of Ukraine
(project 2020.02/0014 ``Asymptotic regimes of perturbed random walks: on the edge of modern and
classical probability''). A. Pilipenko was also partially supported by the Alexander von Humboldt Foundation within
the Research Group Linkage Programme {\it Singular diffusions: analytic and stochastic approaches}.
\end{funding}





\end{document}